\documentclass[10pt]{article}
\usepackage{color}
\usepackage{amsmath}
\usepackage{amssymb}
\usepackage{mathrsfs}
\usepackage[utf8]{inputenc}
\usepackage[english]{babel}
\usepackage{amsthm}
\usepackage{genyoungtabtikz}
\usepackage{tikz}

\usepackage{graphicx}
\graphicspath{{./images/}}

\usepackage[toc,title,page]{appendix}

\usepackage{epsfig}

\usepackage[scaled=1]{helvet}
\usepackage{titlesec}
\usepackage{multido}
\usepackage{multirow}
\usepackage{blindtext}
\usepackage{color}
\usepackage{lipsum}
\usepackage{mathtools}
\usepackage[vcentermath]{youngtab}
\usepackage{young}
\usepackage[all,ps,dvips,graph]{xy}
\usepackage[misc,geometry]{ifsym}

\usepackage[left=3cm,top=2cm,right=3cm,bottom=2cm]{geometry}
\usepackage{graphicx}
\numberwithin{equation}{subsection}

\let\OLDthebibliography\thebibliography
\renewcommand\thebibliography[1]{
  \OLDthebibliography{#1}
  \setlength{\parskip}{0pt}
  \setlength{\itemsep}{0pt plus 0.3ex}
}

\newtheorem{theorem}{Theorem}

\newtheorem{corollary}[theorem]{Corollary}
\newtheorem{lemma}[theorem]{Lemma}
\newtheorem{remark}[theorem]{Remark}
\newtheorem{example}[theorem]{Example}
\newtheorem{definition}[theorem]{Definition}

\numberwithin{theorem}{section}

\newcommand{\A}{\mathcal{A}}
\newcommand{\aaa}{\mathfrak{a}}

\newcommand{\Basis}{\mathcal{C}_n}
\newcommand{\BasisOne}{\mathcal{C}^{1,\K}_n}
\newcommand{\B}{\mathcal{B}_m}
\newcommand{\BB}{\mathcal{B}^{\F}_{l,m}(\kappan,\een)}
\newcommand{\BBnu}{\mathcal{B}^{\F}_{l,m}(\bnu,\een)}
\newcommand{\Bb}{\pmb{\mathfrak{b}}}
\newcommand{\BBB}{\mathcal{B}_{n+1}}
\newcommand{\bbb}{\mathfrak{b}}
\newcommand{\BBk}{\mathcal{B}_k}
\newcommand{\balpha}{\boldsymbol{\alpha}}
\newcommand{\bbeta}{\boldsymbol{\beta}}
\newcommand{\bgamma}{\boldsymbol{\gamma}}
\newcommand{\bGamma}{\boldsymbol{\Gamma}}
\newcommand{\bbkap}{\boldsymbol{\kappa}}
\newcommand\bbn{\mathbf{b}}
\newcommand\bbs{\mathsf{s}}
\newcommand\bbt{\boldsymbol{\mathsf{t}}}
\newcommand\bbu{\mathsf{u}}
\newcommand\bbv{\mathsf{v}}
\newcommand{\Bchico}{\mathcal{B}_{n-1}}
\newcommand\be{\mathbb{E}}
\newcommand\ben{\boldsymbol{e}}
\newcommand{\belongs}{ \in }
\newcommand{\Belongs}{ \ni}
\newcommand{\Bg}{\pmb{\mathfrak{g}}}
\newcommand\bi{\boldsymbol{i}}
\newcommand{\bif}{{\underline{\boldsymbol{i}}}}
\newcommand{\bjf}{{\underline{\boldsymbol{j}}}}
\newcommand\bj{\boldsymbol{j}}
\newcommand\blambda{{\boldsymbol\lambda}}
\newcommand\bmu{{\boldsymbol\mu}}
\newcommand\brho{\boldsymbol{\rho}}
\newcommand\bzeta{{\boldsymbol\zeta}}
\newcommand\bn{\boldsymbol{n}}
\newcommand\bnu{{\boldsymbol\nu}}
\newcommand\boxbluek{\color{blue}\boldsymbol{[k]}}
\newcommand\boxbluej{\color{blue}\boldsymbol{[j]}}
\newcommand\boxredk{\color{red}\boldsymbol{[k]}}
\newcommand\boxredj{\color{blue}\boldsymbol{[j]}}
\newcommand{\bpartial}{\boldsymbol{\partial}}
\newcommand{\bCpartial}{\mathbf{C}\boldsymbol{\partial}}
\newcommand{\Bs}{\pmb{\mathfrak{s}}}
\newcommand\bS{\Sigma}
\newcommand\bs{\mathbf{s}}
\newcommand{\bT}{\pmb{\mathfrak{t}}}
\newcommand\bt{\mathbf{t}}
\newcommand{\bTI}{ \bT^{-1} (\bT_{\theta}^{\blambda}(1))}
\newcommand{\bTII}{ \bT^{-1} (\bT_{\theta}^{\blambda}(2))}
\newcommand{\bTj}{ \bT^{-1} (\bT_{\theta}^{\blambda}(j))}
\newcommand{\bTn}{ \bT^{-1} (\bT_{\theta}^{\blambda}(n))}
\newcommand\btau{{\boldsymbol\tau}}
\newcommand{\Bu}{\pmb{\mathfrak{u}}}
\newcommand{\Bv}{\pmb{\mathfrak{v}}}
\newcommand\bu{\mathbf{u}}
\newcommand\bv{\mathbf{v}}

\newcommand{\C}{\mathcal{C}_{\rm{YH}}}
\newcommand\calA{\mathcal{A}}
\newcommand\calB{\mathcal{B}}
\newcommand\calC{\mathcal{C}}
\newcommand\calD{\mathcal{D}}
\newcommand\calE{\mathcal{E}}
\newcommand\calF{\mathcal{F}}
\newcommand\calG{\mathcal{G}}
\newcommand\calH{\mathcal{H}}
\newcommand\calL{\mathcal{L}}
\newcommand\calM{\mathcal{M}}
\newcommand\calN{\mathcal{N}}
\newcommand\calO{\mathcal{O}}
\newcommand\calP{\mathcal{P}}
\newcommand\calQ{\mathcal{Q}}
\newcommand\calR{\mathcal{R}}
\newcommand\calS{\mathcal{S}}
\newcommand\calT{\mathcal{T}}
\newcommand\calU{\mathcal{U}}
\newcommand\calV{\mathcal{V}}
\newcommand\calW{\mathcal{W}}
\newcommand\calX{\mathcal{X}}
\newcommand\calY{\mathcal{Y}}
\newcommand\calZ{\mathcal{Z}}
\newcommand\bcalJ{\boldsymbol{\mathcal{J}}}
\newcommand\bcalL{\boldsymbol{\mathcal{L}}}
\newcommand\bcalm{\boldsymbol{\mathcal{m}}}
\newcommand\bcalNT{\boldsymbol{\mathcal{NT}}}
\newcommand\bcalNGAT{\boldsymbol{\mathcal{NT}}}
\newcommand\bcalD{\boldsymbol{\mathcal{D}}}
\newcommand\bcalR{\boldsymbol{\mathcal{R}}}
\newcommand\bcalV{\boldsymbol{\mathcal{V}}}
\newcommand\bcalU{\boldsymbol{\mathcal{U}}}
\newcommand\bcalY{\boldsymbol{\mathcal{Y}}}
\newcommand\bcalW{\boldsymbol{\mathcal{W}}}
\newcommand\bnabla{\boldsymbol{\nabla}}

\newcommand{\catorce}{ 14}
\newcommand{\catorceB}{\color{red} 14}
\newcommand{\CC}{ \mathbb C }
\newcommand{\ccc}{\mathfrak{c}}
\newcommand{\ch}{{\rm char}}
\newcommand{\cincuentacinco}{55}
\newcommand{\cincuentacincoR}{\color{red}55}
\newcommand{\Comp}{{\mathcal Comp}_n}
\newcommand{\cuarentacuatro}{ 44}
\newcommand{\cupdot}{\mathbin{\mathaccent\cdot\cup}}
\newcommand{\Cs}{\overleftarrow{C}}
\newcommand{\Cd}{\overrightarrow{C}}

\newcommand{\Der}{{\rm Der}}
\newcommand{\diez}{10}
\newcommand{\dieciseis}{16}
\newcommand{\diezR}{\color{red}10}
\newcommand{\doce}{12}
\newcommand{\doceB}{\color{blue} 12}
\newcommand{\doceR}{\color{red} 12}

\newcommand{\E}{ {\mathcal E}_n(q)}
\newcommand{\e}{\mathfrak{e}}
\newcommand{\EE}{ {\mathcal E}_n}
\newcommand\een{\mathbf{e}}
\newcommand\es{\mathbbm{s}}
\newcommand{\End}{{\rm End}}
\newcommand\et{\mathbbm{t}}
\newcommand\eu{\mathbbm{u}}
\newcommand\ev{\mathbbm{v}}
\newcommand{\Exp}{ {\rm \bf exp} }

\newcommand{\F}{ { \mathbb F}}
\newcommand{\FF}{ {\mathcal F}_n}

\newcommand{\g}{  \mathfrak{g}}
\newcommand{\gl}{\mathfrak{gl}}

\newcommand{\h}{{h}}
\newcommand{\HH}{ \mathcal{H}_n}
\newcommand{\HHO}{ \mathcal{H}^{\OO}_n}
\newcommand{\HHK}{ \mathcal{H}^{\K}_n}
\newcommand{\HHtwo}{ \mathcal{H}_2}
\newcommand{\HHKtwo}{ \mathcal{H}^{\K}_2}
\newcommand{\HHOtwo}{ \mathcal{H}^{\OO}_2}
\newcommand{\HHKOne}{ \mathcal{H}^{1,\K}_n}

\newcommand{\II}{I_{\mathbf{e}}}
\newcommand{\IIa}{I_{e}}
\newcommand{\id}{{\rm id}}
\newcommand{\ind}{{\rm ind}}
\newcommand{\inv}{{\rm inv}}

\newcommand{\JM}{ \mathcal L }

\newcommand{\K}{\mathcal{K}}
\newcommand{\kk}{\mathcal{K}}
\newcommand{\kkk}{k-1}
\newcommand{\kappan}{\boldsymbol{\kappa}}

\newcommand{\Li}{\mathcal{L}}
\newcommand{\LL}{\mathbb{L}}

\newcommand{\m}{\mathfrak{m}}
\newcommand{\MC}{{ {\rm Comp}}_{l,n}}
\newcommand{\MCm}{{ {\rm Comp}}_{l,m}}
\newcommand{\MP}{{\rm Par }_{l,n}}
\newcommand{\mfra}{{\mathfrak{a}}}
\newcommand{\mfrb}{{\mathfrak{b}}}
\newcommand{\mfrc}{{\mathfrak{c}}}
\newcommand{\mfrt}{{\mathfrak{t}}}
\newcommand{\mfrs}{{\mathfrak{s}}}
\newcommand{\mfru}{{\mathfrak{u}}}
\newcommand{\mfrv}{{\mathfrak{v}}}
\newcommand{\mfrw}{{\mathfrak{w}}}
\newcommand{\mfrx}{{\mathfrak{x}}}
\newcommand{\mfry}{{\mathfrak{y}}}
\newcommand{\mfrz}{{\mathfrak{z}}}

\newcommand{\N}{ { \mathbb N}}
\newcommand{\No}{ { \mathbb N}_0}
\newcommand{\nstd}{{\rm NStd}}
\newcommand{\NB}{\mathbb{N}\mathcal{B}_k}

\newcommand{\once}{11}
\newcommand{\onceB}{\color{blue}11}
\newcommand{\OnePar}{{ \rm Par}^1_{l,m}}
\newcommand{\OneParn}{{ \rm Par}^1_{l,n}}
\newcommand{\OO}{\mathcal{O}}
\newcommand{\op}{\otimes}

\newcommand{\Par}{{\rm Par}_{l,m}}

\newcommand{\q}{\hat{q}}
\newcommand{\quince}{15}
\newcommand{\quinceB}{\color{red} 15}

\newcommand{\R}{ \mathcal{R}_m}
\newcommand{\Rad}{{\rm Rad}}
\newcommand{\res}{ \textrm{res} }
\newcommand\rrn{\mathbf{r}}

\newcommand{\s}{\mathfrak{s}}
\newcommand{\seq}{{\rm seq}_n}
\newcommand{\shape}{\textsf{shape}}
\newcommand{\Si}{\mathfrak{S}}
\newcommand{\sign}{-}
\newcommand{\sixteenB}{\color{red} 16}
\newcommand{\Snake}{{\rm Snake}}
\newcommand{\spa}{{\rm span}}
\newcommand{\std}{{\rm Std}}

\newcommand{\T}{  \mathfrak{t}}
\newcommand{\tab}{{\rm Tab}}
\newcommand{\tr}{{\rm \textbf{tr}}}
\newcommand{\tR}{ \overline{\R}}
\newcommand{\Tr}{{\rm Tr}}
\newcommand{\trece}{13}
\newcommand{\treintatres}{33}
\newcommand{\treintatresR}{\color{red}33}
\newcommand{\trunc}{ {\B} (\blambda ) }
\newcommand{\truncPrime}{ {\mathbb B}_n^{\prime} (\blambda ) }
\newcommand{\truncSing}{ {\mathbb B}_{\bar n} (\overline{\blambda} ) }
\newcommand{\TT}{{\mathfrak T}}
\newcommand{\TTc}{  \mathcal{T}}

\newcommand{\U}{\mathfrak{u}}
\newcommand{\UU}{\mathbb{U}}
\newcommand{\UUc}{\mathcal{U}}
\newcommand{\ulu}{\underline{u}}
\newcommand{\ulv}{\underline{v}}
\newcommand{\ulw}{\underline{w}}
\newcommand{\ulx}{\underline{x}}
\newcommand{\uly}{\underline{y}}
\newcommand{\ulz}{\underline{z}}

\newcommand{\V}{\mathfrak{v}}
\newcommand{\veintidos}{22}
\newcommand{\VV}{\mathbb{V}}

\newcommand{\Y}{\mathcal{Y}}
\newcommand{\Yy}{\mathcal{Y}_{r,n}}
\newcommand{\yvc}{\Yvcentermath1}
\newcommand{\YY}{\mathcal{Y}_{r,n}(q)}

\newcommand{\Z}{\mathbb{Z}}

\begin{document}
  \Yvcentermath1
\title{Nil graded algebras associated to triangular matrices and their applications to Soergel Calculus.}
\author{\sc  Diego Lobos Maturana }

%\date{\today}
\maketitle

%\pagenumbering{roman}

\begin{abstract}
  We introduce and study a category of algebras strongly connected with the structure of the Gelfand-Tsetlin subalgebras of the endomorphism algebras of Bott-Samelson bimodules. We develop a series of techniques that allow us to obtain optimal presentations for the many Gelfand-Tsetlin subalgebras appearing in the context of the Diagrammatic Soergel Category.
\end{abstract}
%\newpage
keywords: Diagrammatic Soergel Category, Gelfand-Tsetlin subalgebras.
\tableofcontents

%\newpage
\pagenumbering{arabic}

\section{Introduction}
\subsection{Motivation}
The \emph{Hecke category} associated to a Coxeter group, is currently one of the most relevant object of study in Representation Theory, due to its multiple applications to Geometric and Modular representation theory among other branches of study (see \cite{bow-cox-hazi}, \cite{EliasKhov10}, \cite{EW}, \cite{Esp-Pl},  \cite{JenWill17}, \cite{Lib10}, \cite{Lib15}, \cite{LibLightLeaves}, \cite{LiPl}, \cite{Lobos-Plaza-Ryom-Hansen}, \cite{Plaza17}, \cite{RicheWill}, \cite{Steen2} for example).

Thanks to the work of Elias and Williamson \cite{EW}, we have the \emph{Diagrammatic Soergel category} $\bcalD,$ a diagrammatic approach to the algebraic study of the Hecke category. Libedinsky in \cite{LibLightLeaves} has built the well known \emph{Double light leaves} basis for the many morphism spaces of the Soergel Category and later Elias and Williamson have proved in \cite{EW}, that those bases are cellular in the sense of \cite{Westbury}. In particular, due to the work of Ryom-Hansen \cite{Steen2}, we know that the many endomorphisms algebras coming from $\bcalD$ are cellular algebras endowed with a set of Jucys-Murphy elements in the sense of \cite{Mat-So}. In this article we study the structure of the \emph{Gelfand-Tsetlin} subalgebras in the category $\bcalD,$ generated by the Jucys-Murphy elements obtained by Ryom-Hansen.

 Jucys-Murphy elements, when they exist, are fundamental in the study of the representation theory of cellular algebras (see \cite{Murphy1},\cite{Murphy2}, \cite{Murphy3}, \cite{MatCoef},\cite{Mat-So}, \cite{OkunVershik2}, for example). In this case we are particularly interested in the search of optimal presentations in terms of minimal set of generators and relations for the subalgebras generated by those elements.

In the beginning of our study,  we observed that the many Gelfand-Tsetlin subalgebras in the context of the diagrammatic Soergel category, have basically the same structure independently of the type where we were immersed. Although they are not isomorphic in general, there were certain patterns appearing in first presentations that we obtained, that invited us to provide a more general approach. Thereby we define a category $\bcalNT,$ whose objects are the \emph{Nil graded algebras associated to triangular matrices}, the main object of study in this article, and whose morphisms are the \emph{preserving degree} homomorphisms between them. The category $\bcalNT$ is by itself an interesting object of study due to its intriguing structure, far beyond the applications that we develop in this work.

In this article we provide a first series of generalities about the category $\bcalNT,$ that we expect to develop more completely in future works. Although we left many questions open in the study of the category $\bcalNT,$ for our original purposes, we obtain sufficient and very effective techniques that allow us to obtain the desired optimal presentations, explicit bases and dimension formula for all the Gelfand-Tsetlin subalgebras coming from Soergel calculus.
%A nil graded algebra associated to a triangular matrix is basically a commutative graded algebra $\calA(T),$ defined by an abstract presentation whose relations are codifying by a strictly lower matrix $T$ (see definition \ref{def-nilalgebra-associated-to-T}).

This article is, in some way, a generalization of the work of Espinoza and Plaza \cite{Esp-Pl}, where the authors obtained optimal presentations for the Gelfand-Tsetlin subalgebras, in type $\widetilde{A}_1,$ named by the authors as the \emph{Dot-line algebras}.  In our case we extend their results working simultaneously in any type. In particular, we provide an explicit treatment to the case of type $\widetilde{A}_m,(m\geq2).$

%\subsection{Structure of the paper}

The outline of this paper is as follows: In section \ref{sec-Nil-alg-associated-to-T} we define the concept of \emph{nil graded algebra associated to a triangular matrix} (definition \ref{def-nilalgebra-associated-to-T}), and we develop a series of preliminary results concerning with their structure. In particular in Theorem \ref{theo-all-AT-is-strongly} we obtain complete control on their vector space structure, that is, we build bases and we obtain explicit dimension formulas. Later in subsection \ref{sec-the-category-NT} we introduce and study the category $\bcalNT,$ we define their objects and morphisms and we provide a series of small results on its structure.

In section \ref{sec-Applic-Soergel} we develop the connection  of the category $\bcalNT$ with Soergel Calculus. We first briefly recall some of the elements involved in the work of Elias and Williamson \cite{EW}, focusing basically in those that will be used in this article (subsection \ref{ssec-Soergel-category}). After that we explicitly related each object $\ulw$ of the category $\bcalD,$ with an object $\calA_{\ulw}$ in the category $\bcalNT$ (subsection \ref{ssec-nil-alg-and-expressions}). In subsection \ref{ssec-the-Jailbreak-alg} we define the Gelfand-Tsetlin subalgebra $\bcalJ(\ulw)$ for any endomorphism algebra $\textrm{End}_{\bcalD}(\ulw)$ in the category $\bcalD.$  We show that they are members of our category $\bcalNT.$ Moreover we provide an isomorphism between the algebras $\bcalJ(\ulw)$ and $\calA_{\ulw}$ (see Theorem \ref{theo-Ju-isomorphic-to-ATu}). Finally in subsection \ref{ssec-explicit-for-type-A} we turn our attention to the case of type $\widetilde{A}_m.$ We show more explicitly how the technology coming from the category $\bcalNT$ can be used to optimally obtain presentations for the Gelfand-Tsetlin subalgebras corresponding to this type.

A particular result of this article is Theorem \ref{theo-isomorph-for-vertical-case}, where we obtain an explicit isomorphism between certain Gelfand-Tsetlin subalgebras in the category $\bcalD$ with Gelfand-Tsetlin subalgebras in the context of Generalized Blob algebras, that were studied by the author in \cite{Lobos1}. This isomorphism is coherent with the isomorphism of categories described in the \emph{Blob vs Soergel} conjecture of Libedinsky and Plaza \cite{LiPl} (recently proved in \cite{bow-cox-hazi}) where the authors settle isomorphisms between endomorphims algebras of Bott-Samelson bimodules and certain idempotent truncations of generalized blob algebras. Although we cannot conclude from this conjecture that the Gelfand-Tsetlin subalgebras will be necessarily isomorphic, since they are relative to the cellular bases we are working with. In this case, under the particular hypothesis of theorem \ref{theo-isomorph-for-vertical-case} we obtain this connection.

Summarizing, the main results of this article are:
\begin{enumerate}
  \item Theorem \ref{theo-all-AT-is-strongly}, where we provide fundamental information on the structure of the algebras $\calA(T)$ of our category $\bcalNT.$
  \item Theorem \ref{theo-Ju-isomorphic-to-ATu}, where we describe how we can calculate explicitly, the desired  presentations for the many Gelfand-Tsetlin subalgebras, coming from the Soergel Calculus in any type.
  \item Theorem \ref{lemma-commuting-move-for-type-Atilde}, where we show that the Gelfand-Tsetlin subalgebras $\bcalJ(\ulu)$ in type $\widetilde{A},$ are invariant under commuting moves. This theorem is the base for our explicit algorithm that allow us to calculate the desired presentations for this case.
  \item Theorem \ref{theo-isomorph-for-vertical-case}, in this case, a strong demonstration on how our results can be used in a deeper study of the strong connection between Generalized blob algebras and Soergel Calculus.
\end{enumerate}

%At the end of this article we add some appendices to briefly explain some basic facts about \emph{Graded algebras} (appendix \ref{ssec-nil-graded-alg}), Coxeter groups (appendix \ref{ssec-coxeter-groups}) and the \emph{Diagrammatic Soergel Category} (appendix \ref{ssec-Soergel-category}). Their reading is optional but we refer them in some points of this article, to recall some properties needed for our work.

\subsection{Notations and terminology}\label{ssecnotations}
The following is a list of notation and terminology that will be used along this article:
%We invite the readers not familiarized with the topics
\begin{itemize}
\item For the whole article we consider $\F$ as a field of characteristic different from $2.$  We denote  $\F^{\times}=\F-\{0\}.$  Any vector space (resp. algebra) mentioned in this article will be considered over the field $\F.$ Any matrix mentioned in this article will be considered with coefficient in the field $\F.$ Any linear transformation will be understood as a $\F-$linear transformation.

\item We denote by $\mathbb{Z}$ the set of integer numbers and $\mathbb{Z}_{\geq 0}$ the set of nonnegative integer numbers. If $m\in\mathbb{Z},$ we denote by $m\mathbb{Z}$ the set of all integer numbers multiples of $m.$ We also denote by $I_m=\mathbb{Z}/m\mathbb{Z},$ the ring of classes modulo $m.$ The class modulo $m$ of a number $a\in \mathbb{Z}$  will be denoted simply as $a.$

\item In this article we use the word \emph{algebra} to refer to any unital, associative algebra (over $\F$). The identity element of an algebra $\calA$ will be denoted by $1_\calA$ or simply by $1$ if there is no possible confusion.
    If $\calA,\calB$ are algebras, with $\calB\subset \calA,$ we say that $\calB$ is a \emph{subalgebra} of $\calA$ if they share the same identity element $(1_\calB=1_\calA).$ In other case we only say that $\calB$ is an \emph{algebra contained} in $\calA.$ Any homomorphism of algebras $\bgamma:\calA\rightarrow \calB$ will be assumed satisfying the condition $\bgamma(1_\calA)=1_\calB.$ Therefore the set $\bgamma(\calA)$ will be always a subalgebra of $\calB.$
    We call \emph{the vector space structure} of an algebra $\calA,$ the vector space obtained from $\calA$ by forgetting the product operation. When it would be necessary, we denote the vector space structure of $\calA$ as $\calA_{v}.$

\item  A \emph{preserving degree} homomorphism of graded algebras $\bgamma:\calA\rightarrow \calB$  is an homomorphism of algebras between two graded algebras $\calA,\calB,$ such that $\deg(\bgamma(h))=\deg(h),$ for any homogenous element $h$ of $\calA.$  A \emph{Nil graded algebra} is a graded algebra $\calA,$ where every homogeneous element $h$ such that $\deg(h)\neq 0$ is a nilpotent element.

\item The \emph{Gelfand-Tsetlin} subalgebra of a cellular algebra endowed by a set of Jucys-Murphy elements is the subalgebra that those elements generate (this terminology was introduced by Okounkov and Vershik in \cite{OkunVershik1},\cite{OkunVershik2}).

\item Let $(W,S)$ a Coxeter system, relative to a set of indexes $J$ and Coxeter matrix $\mathbf{m}$.  Let $W^{\ast}$ the free monoid  generated by $S=\{s_a:a\in J\}.$ The elements $\underline{w}\in W^{\ast}$ are \emph{free expressions} (or simply \emph{expressions}) on the elements on $S.$ Each expression $\ulw\in W$ defines an element in the group $W$, that we will denote by $w$ (sometimes if there is no possibly confusion, we will denote both the expression in $W^{\ast}$ and the element in $W$ with the same symbol). The length $l^{\ast}(\underline{w})$ of the expression $\underline{w}\in W^{\ast},$ is the number of factors in $S$ that compose it. The length $l(w)$ of an element $w\in W$ is given by  $l(w)=\min\{l^{\ast}(\underline{w}):\ulw\},$ where $\ulw\in W^{\ast}$ runs over all the possible expressions of $w.$  When $\l^{\ast}(\underline{w})=l(w),$ we say that $\underline{w}$ is a \emph{reduced expression} for $w.$ In particular if $(W,S)$ is the Coxeter system of type $\tilde{A}_{m-1},$ we consider as a set of indexes $J=I_m=\mathbb{Z}/m\mathbb{Z}$ (more details in \cite{Bjorner-Brenti}).
   % If $\mathbf{m}(a,b)$ is the order of the element $s_as_b\in W$ then we write:
%\begin{equation*}
%  B(s_a,s_b)=s_as_b\cdots\quad \textrm{and}\quad B(s_b,s_a)=s_bs_a\cdots\quad (\mathbf{m}(a,b)-\textrm{factors})
%\end{equation*}
%Then $B(s_a,s_b)$ and $B(s_b,s_a)$ can be seen as (different) expressions in $W^{\ast}$ or as an element (the same) in $W.$
\end{itemize}

\subsection{Acknowledgment}
\medskip

 It is a pleasure to thank to David Plaza and Steen Ryom-Hansen for many useful comments that helped to improve this article. I especially want to thank go to Camilo Martinez Estay, bachelor student in Mathematics at PUCV, for his excellent work as Assistant Researcher in the first part of this project.

 This work was supported by \emph{Proyecto DI Emergente PUCV 2020}, 039.475/2020. Pontificia Universidad Cat\'olica de Valpara\'iso, Chile.
\medskip

%It is a pleasure to thank

\medskip

\section{Nil graded algebras associated to triangular matrices}\label{sec-Nil-alg-associated-to-T}

\subsection{Nil graded algebras weakly associated to triangular matrices}\label{ssec-nil-graded-alg-weakly-associated-to-T}
In this section we define a family of algebras, whose presentation is given by a strictly lower triangular matrix

\begin{definition}\label{def-nilalgebra-associated-to-T}
  Given a strictly lower triangular matrix $T=[t_{ij}]_{n\times n},$ we define the \emph{Abstract nil graded algebra associated to} $T$ as the commutative unital $\F-$algebra $\calA(T)$ given by generators $X_1,\dots,X_n$ subject to the following relations:
  \begin{equation}\label{rel-nilalgebra-associated-to-T}
    X_{1}^{2}=0,\quad \textrm{and}\quad X_{i}^{2}=\sum_{j<i}{t_{ij}}X_{j}X_i,\quad (i=2,\dots,n).
  \end{equation}
  We provide a structure of positively graded algebra for $\calA(T),$ by defining
  \begin{equation}\label{rel-def-degree-AT}
   \deg(X_i)=2,\quad (i=1,\dots,n).
  \end{equation}
  %and
  %\begin{equation}\label{rel-def-degree-AT-2}
  % \deg(1)=0.
  %\end{equation}
\end{definition}

%Given any $1\leq m<n$ we define the \emph{restricted matrix} $T|_m$ as the matrix obtained from $T$ by deleting rows and columns from $m+1$ to $n.$ The matrix $T|_m$ is also a strictly lower triangular then it determines the algebra $\calA(T|_m).$ We can see the algebra $\calA(T|_m)$ as the subalgebra of $\calA(T)$ generated by the elements $X_1,\dots,X_m.$

%More generally, we can define a weak version of definition \ref{def-nilalgebra-associated-to-T} as follows:

\begin{definition}\label{def-nilalgebra-weakly-associated-to-T}
  Let $T=[t_{ij}]_{n\times n}$ be a lower triangular matrix. A commutative graded $\F-$algebra $\calA$ generated by a set of elements $\calX=\{X_1,\dots,X_n\},$ is \emph{weakly associated to} $T$ (relative to the set $\calX$), if its generators $X_1,\dots,X_n$ satisfy relations \ref{rel-nilalgebra-associated-to-T} and \ref{rel-def-degree-AT}.
\end{definition}

Note that the difference between definitions \ref{def-nilalgebra-associated-to-T} and  \ref{def-nilalgebra-weakly-associated-to-T} is that in definition \ref{def-nilalgebra-associated-to-T} the abstract algebra $\calA(T)$ is defined by a presentation, while in definition \ref{def-nilalgebra-weakly-associated-to-T} we consider any commutative graded algebra $\calA$ whose generators satisfy relations \ref{rel-nilalgebra-associated-to-T}. In particular the abstract algebra $\calA(T)$ of definition \ref{def-nilalgebra-associated-to-T} is also an algebra weakly associated to $T.$

%If $\calA$ is an algebra weakly associated to the strictly lower triangular matrix $T=[t_{ij}]_{n\times n},$ relative to  set of generators $\calY=\{Y_1,\dots,Y_n\}$ (definition \ref{def-nilalgebra-weakly-associated-to-T}),  then there is a natural preserving degree epimorphisms of algebras $\bgamma:\calA(T)\rightarrow \calA,$ given by $X_i\mapsto Y_i.$

The following lemmas and corresponding corollaries are a generalization of technics developed in \cite{Esp-Pl} and \cite{Lobos1}:

\begin{lemma}\label{lemma-bounded-dimension-for-AT}
  Let $\calA$ an algebra weakly associated to a strictly lower triangular matrix $T=[t_{ij}]_{n\times n},$ relative to a set of generators $\calX=\{X_1,\dots,X_n\}.$
   Then the vector space structure of $\calA,$ is generated by the set:
  $$W_{0}=\{X_1^{\alpha_1}X_2^{\alpha_2}\cdots X_{n}^{\alpha_n}:\alpha_{j}\in\{0,1\}\}$$
  (note that $1=X_1^{0}X_2^{0}\cdots X_{n}^{0}\in W_{0}$).
\end{lemma}
\begin{proof}
  It is clear that as a vector space, $\calA$ is generated by the set:
  $W=\{X_1^{\alpha_1}X_2^{\alpha_2}\cdots X_{n}^{\alpha_n}:\alpha_{j}\in\mathbb{Z}_{\geq 0}\}.$
  Let $W_1=\{X_1^{\alpha_1}X_2^{\alpha_2}\cdots X_{n}^{\alpha_n}:\alpha_{i}\geq 2\quad \textrm{for some}\quad i=1,\dots,n \},$  then $W=W_0\cup W_1,$ therefore, it is enough to prove that each element of $W_1$ can be written as a linear combination of elements in $W_0.$

  For each element $P=X_1^{\alpha_1}X_2^{\alpha_2}\cdots X_{n}^{\alpha_n}$ in $W_1$ we define the number $m(P)=\min\{i:\alpha_i\geq 2\}.$ If $m(P)=j,$ we write $h(P)=\alpha_j.$
  We define a partial order on $W_1$ as follows:
  \begin{equation*}
    P_1\prec P_2\quad\textrm{if and only if}\quad (m(P_1)<m(P_2))\vee (m(P_1)=m(P_2)\wedge h(P_1)<h(P_2)).
  \end{equation*}
  We use induction on $\prec$ to prove that each element of $W_1$ can be written as a linear combination of elements in $W_0.$

  If $m(P)=1,$ then by relation \ref{rel-nilalgebra-associated-to-T} we have $P=0$ and the assertion is trivial.
  If $m(P)=i>1$ then by relation \ref{rel-nilalgebra-associated-to-T} we have
  $P=X_1^{\alpha_1}X_2^{\alpha_2}\cdots X_{i-1}^{\alpha_{i-1}}HX_i^{\alpha_i-2}\cdots X_{n}^{\alpha_n},$ where $H=\sum_{j<i}{t_{ij}X_jX_i}.$ Therefore $P=\sum_{j<i}t_{ij}R_j,$
  where $ R_j=X_1^{\alpha_1}X_2^{\alpha_2}\cdots X_{j}^{\alpha_{j}+1}\cdots X_i^{\alpha_i-1}\cdots X_{n}^{\alpha_n}.$
  Then either $R_j\in W_0$ or $R_j\in W_1$ and satisfies $R_j\prec P.$ In the second case by induction we have that $R_j$ can be written as a linear combination of elements in $W_0.$ Both cases imply that $P$ can be written as a linear combination of elements in $W_0$ as desired.
\end{proof}

\begin{corollary}\label{coro-bounded-degree-for-AT}
  Let $\calA$ an algebra weakly associated to a strictly lower triangular matrix $T=[t_{ij}]_{n\times n},$ relative to the set $\calX=\{X_1,\dots,X_n\}.$ Then
  \begin{enumerate}
    \item $ \dim(\calA)\leq 2^{n}.$
    \item If $h$ be an homogeneous element in $\calA,$ then $\deg(h)\leq 2n.$ In particular, if $\deg(h)=2n,$ then
  \begin{equation*}
    h=rX_1X_2\cdots X_n,\quad\textrm{for some}\quad r\in \F.
  \end{equation*}
  \end{enumerate}

\end{corollary}
\begin{proof}
  It follows directly from lemma \ref{lemma-bounded-dimension-for-AT}.
\end{proof}

\begin{lemma}\label{lemma-monomial-anihilation-v1}
 Let $\calA$ an algebra, weakly associated to a strictly lower triangular matrix $T=[t_{ij}]_{n\times n},$ relative to the set $\calX=\{X_1,\dots,X_n\}.$ For each $i=1,\dots,n$ we have $X_1X_2\cdots X_{i-1}X_i^2=0.$
\end{lemma}
\begin{proof}
  If $i=1$ the assertion follows directly from relation \ref{rel-nilalgebra-associated-to-T}. If $i>1$ then by relation \ref{rel-nilalgebra-associated-to-T} we have $ X_1X_2\cdots X_{i-1}X_i^2=\sum_{j<i}{t_{ij}}X_1X_2\cdots X_{j}^2\cdots X_i,$ therefore we obtain the desired result by induction on $i$ .
\end{proof}

\begin{lemma}\label{lemma-nilpotency-of-Xi}
  Let $\calA$ an algebra, weakly associated to a lower triangular matrix $T=[t_{ij}]_{n\times n},$ relative to the set $\calX=\{X_1,\dots,X_n\}.$ Then $X_i^{i+1}=0,$ for each $i=1,\dots n.$
\end{lemma}
\begin{proof}
  If $i=1$ the assertion follows directly from relation \ref{rel-nilalgebra-associated-to-T}. If $i>1$ then by relation \ref{rel-nilalgebra-associated-to-T} we have $X_i^{i+1}=\left(\sum_{j<i}{t_{ij}X_j}\right)^{i-1}X_i^{2}.$ The element $h=\left(\sum_{j<i}{t_{ij}X_j}\right)^{i-1}$ is an homogeneous element of degree $2(i-1)$ in the subalgebra $\calA(T|_{i-1})$ of $\calA(T).$   Then by corollary \ref{coro-bounded-degree-for-AT} we have $h=rX_1\cdots X_{i-1}$ and therefore by lemma \ref{lemma-monomial-anihilation-v1} we have $X_{i}^{i+1}=0$ as desired.
  (Here $T|_{i-1}=[t_{ij}]$ is the submatrix of $T$ obtained by erasing the rows and columns of $T$ from $i$ to $n.$)
\end{proof}

%The last lemma implies that any algebra $\calA,$ weakly associated to a strictly lower triangular matrix, is in fact a nil graded algebra in the sense of definition \ref{def-nil-graded-algebras} (see appendix \ref{ssec-nil-graded-alg}).

\subsection{Nil graded algebras strongly associated to triangular matrices}\label{ssec-nil-alg-strongly-associated-T}
\begin{definition}\label{def-strongly-AT}
  Given a strictly lower triangular matrix $T=[t_{ij}]_{n\times n},$ we say that an algebra $\calA,$ weakly associated to $T$  is \emph{strongly associated} to $T$, if $\dim(\calA)=2^n.$
\end{definition}

Note that by definition, any algebra $\calA$ strongly associated to a lower triangular matrix $T$, is also weakly associated to $T.$ If the algebra $\calA$ is \emph{strongly associated} to $T$, then the set $W_{0}$ of lemma \ref{lemma-bounded-dimension-for-AT} is a basis for $\calA.$ In that case, we called the set $W_{0},$ the \emph{monomial basis} of the algebra $\calA.$

The following lemma provides a criterion to identify if the algebra $\calA$ is strongly associated to $T$ or not. It is also a generalization of arguments used in \cite{Esp-Pl} and \cite{Lobos1}:

\begin{lemma}\label{lemma-monomial-basis}
   Let $\calA$ an algebra, weakly associated to a strictly lower triangular matrix $T=[t_{ij}]_{n\times n},$ relative to the set $\calX=\{X_1,\dots,X_n\}.$
  The algebra $\calA$ is strongly associated to $T$ if and only if the element $\calP=X_1X_2\cdots X_n$ is different from zero.
\end{lemma}

\begin{proof}
  If $\calA$ is strongly associated to $T$ then the set $W_0$ is a basis for $\calA.$ In particular the element  $\calP=X_1X_2\cdots X_n\in W_0$ is different from zero.

  On the other hand, if $\calP=X_1X_2\cdots X_n\neq 0,$ then each element $\calQ=X_1^{\alpha_1}X_2^{\alpha_2}\cdots X_n^{\alpha_n}\in W_0$ is different from zero. For each of them we define the element $\calF(\calQ)=X_1^{\overline{\alpha}_1}X_2^{\overline{\alpha}_2}\cdots X_n^{\overline{\alpha}_n}\in W_0,$ where $\overline{\alpha}_j=1-\alpha_j.$ Note that $\calQ\calF(\calQ)=\calP\neq 0.$

  We also define for each $\calQ=X_1^{\alpha_1}X_2^{\alpha_2}\cdots X_n^{\alpha_n}\in W_0$ the number $$f(\calQ)=\left\{\begin{array}{cc}
                                                        \min\{j:\alpha_j=1\} & \quad\textrm{if}\quad \calQ\neq 1 \\
                                                        \quad & \quad \\
                                                        0 & \quad\textrm{if}\quad \calQ=1.
                                                      \end{array}\right.$$

  For two elements $\calQ, \calR\in W_0$ we define the following order relation:
  \begin{equation*}
    \calQ<\calR\quad\textrm{if and only if}\quad \left((\deg(\calQ)<\deg(\calR))\vee (\deg(\calQ)=\deg(\calR)\wedge f(\calQ)>f(\calR))\right).
  \end{equation*}
Considering the equation:
\begin{equation}\label{eq1-lemma-monomial-basis}
  \sum_{\calQ\in W_0}{C_{\calQ}\calQ}=0,
\end{equation}
we multiply by $\calP,$ then  by lemma \ref{lemma-monomial-anihilation-v1} we obtain $C_{1}\calP=0,$ and therefore $C_{1}=0.$

Now we take an arbitrary $\calQ\neq 1$ in $W_0$ and assume that we already have proved that $C_{\calR}=0$ for all $\calR<\calQ$ in $W_0.$ Then multiplying by $\calF(\calQ)$ in equation \ref{eq1-lemma-monomial-basis}, then we obtain

\begin{equation}\label{eq2-lemma-monomial-basis}
  C_{\calQ}\calP+\sum_{\calR>\calQ}{C_{\calR}\calR\calF(\calQ)}=0
\end{equation}
but note that $\deg(\calR\calF(\calQ))>2n$ if $\deg(\calR)>\deg(\calQ)$ and then $\calR\calF(\calQ)=0.$ (corollary \ref{coro-bounded-degree-for-AT}). On the other hand if $\deg(\calR)=\deg(\calQ),$ but $f(\calR)<f(\calQ)$ then $\calR\calF(\calQ)=0$ as a consequence of lemma \ref{lemma-monomial-anihilation-v1}. Therefore equation \ref{eq2-lemma-monomial-basis} reduces to $C_{\calQ}\calP=0, $ and then $C_{\calQ}=0.$

By induction on $<,$ we conclude that the set $W_0$ is linearly independent, and by lemma \ref{lemma-bounded-dimension-for-AT}, we conclude that $W_0$ is in fact a basis for $\calA.$ In particular, we conclude that $\calA$ is strongly associated to $T$.
\end{proof}

\begin{theorem}\label{theo-all-AT-is-strongly}
  For any strictly lower triangular matrix $T=[t_{ij}]_{n\times n},$ the abstract nil graded algebra associated $\calA(T)$ (definition \ref{def-nilalgebra-associated-to-T}), is strongly associated to $T$.
\end{theorem}
\begin{proof}
  Let $\calM=\calM_{2n\times 2n}(\F)$ the set of all the matrices of size $2n\times 2n$ and coefficient in $\F.$  Let $\bcalV$ the set of all functions $f:\calM\rightarrow\F.$ For any matrix $H\in\calM$ we define the function $f_{H}:\calM\rightarrow\F$ given by
  \begin{equation*}
    f_{H}(X)=\left\{\begin{array}{cc}
                      1 &\quad\textrm{if}\quad X=H \\
                      \quad & \quad \\
                      0 & \quad \textrm{otherwise}
                    \end{array}\right.
  \end{equation*}
  Then the set $\{f_{H}:H\in \calM\}$ is a basis of $\bcalV.$ If we identify each matrix $H\in\calM$ with its corresponding function $f_{H}\in \bcalV,$ then we can see $\bcalV$ as the vector space of all the formal linear combinations between the elements of $\calM.$

  For each $j=1,\dots,n$ we define the matrix $H(j)=[h_{rc}]\in \calM$ as the matrix whose coefficients are given by the rule:

  \begin{equation*}
    h_{rc}=\left\{\begin{array}{cc}
             1 & \quad\textrm{if}\quad (r,c)=(2j,2j-1) \\
             \quad & \quad \\
             0 & \quad\textrm{otherwise}
           \end{array}\right.
  \end{equation*}
  More generally, for any nonempty subset $J\subset\{1,\dots,n\}$ we define the matrix $H(J)=[h_{rc}]\in \calM$ as the matrix whose coefficients are given by the rule:

  \begin{equation*}
    h_{rc}=\left\{\begin{array}{cc}
             1 & \quad\textrm{if}\quad (r,c)=(2j,2j-1)\quad \textrm{for some}\quad j\in J \\
             \quad & \quad \\
             0 & \quad\textrm{otherwise}
           \end{array}\right.
  \end{equation*}

  We also define the matrix $H(\emptyset)\in \calM$ as the zero matrix $0\in \calM$ (note that this is not the zero element $\overrightarrow{0}$ of the space $\bcalV$).

   Let $\bcalU$ the subspace of $\bcalV$ generated by all the $H(J)$ defined as above. Is not difficult to see that $\dim(\bcalU)=2^n.$

   Let's define a product in $\bcalU:$

    For any pair of disjoint subsets $J_1,J_2\subset \{1,\dots,n\},$ we define

   \begin{equation}\label{eq8-theo-all-AT-is-strongly}
     H(J_1)H(J_2)=H(J_1 \cup J_2).
   \end{equation}

   In particular for each pair $i,j\in\{1,\dots,n\}$ with $i\neq j$ we have

   \begin{equation}\label{eq4-theo-all-AT-is-strongly}
     H(i)H(j)=H(\{i,j\}).
   \end{equation}

   For each $i\in\{1,\dots,n\}$ we define
   \begin{equation}\label{eq5-theo-all-AT-is-strongly}
     H(i)^2=H(i)H(i)=\sum_{j<i}t_{ij}H(\{j,i\}).
   \end{equation}

   We extend by linearity and associativity the product to any element in $\bcalU.$ In particular, we have:

   For each pair of subsets $J_1,J_2\subset \{1,\dots,n\}:$
   \begin{equation}\label{eq6-theo-all-AT-is-strongly}
     H(J_1)H(J_2)=H(J_1 \Delta J_2)H(J_1\cap J_2)^2.
   \end{equation}
   where $J_1\Delta J_2=(J_1\cup J_2) -( J_1\cap J_2),$ is the symmetric difference of the sets $J_1,J_2,$ and where

   \begin{equation}\label{eq7-theo-all-AT-is-strongly}
     H(I)^2=\prod_{i\in I}{H(i)^2}.
   \end{equation}

   for any subset $I\subset\{1,\dots,n\}.$

   It is clear from equations \ref{eq8-theo-all-AT-is-strongly}, \ref{eq4-theo-all-AT-is-strongly}, \ref{eq5-theo-all-AT-is-strongly}, \ref{eq6-theo-all-AT-is-strongly} and \ref{eq7-theo-all-AT-is-strongly} that $H(J_1)H(J_2)\in \bcalU$ for any pair of subsets $J_1,J_2\subset \{1,\dots,n\}.$

    By equation \ref{eq6-theo-all-AT-is-strongly} we conclude that $H(J_1)H(J_2)=H(J_2)H(J_1),$ then $\bcalU$ is commutative under this product (by definition $\bcalU$ is associative under this product).

     By  equation \ref{eq8-theo-all-AT-is-strongly}, we have that $H(\emptyset)H(J)=H(J)=H(J)H(\emptyset),$ then $H(\emptyset)$ is the identity element for this product.

     Therefore $\bcalU$ is a commutative algebra under the aforementioned product and by definition $\dim(\bcalU)=2^n.$

     On the other hand,  equations \ref{eq8-theo-all-AT-is-strongly}, \ref{eq4-theo-all-AT-is-strongly} and \ref{eq5-theo-all-AT-is-strongly}, implies that there is an epimorphism of algebras, from $\calA(T)$ onto $\bcalU,$ given by the assignment:
     \begin{equation*}
       \begin{array}{cc}
         1\mapsto & H(\emptyset) \\
         \quad& \quad \\
         X_i\mapsto & H(i)
       \end{array}
     \end{equation*}

     By lemma \ref{lemma-bounded-dimension-for-AT} we conclude that $\dim(\calA(T))=2^n,$ (an therefore the algebras $\calA(T)$ and $\bcalU$ are isomorphic), then $\calA(T)$ is strongly associated to $T.$
\end{proof}

Note that we can provide a grading to the algebra $\bcalU$ of the proof of theorem \ref{theo-all-AT-is-strongly}, as follows:
\begin{equation*}
  \deg(H(J))=2\cdot|J|
\end{equation*}
where $|J|$ is the cardinal number of the subset $J\subseteq\{1,\dots,n\}.$

%The last theorem implies that the set $W_0$ of lemma \ref{lemma-bounded-dimension-for-AT} is always a basis of $\calA(T).$
%If we consider the decomposition:
%\begin{equation*}
%  \calA(T)=\bigoplus_{g\in \mathbb{Z}}\calA(T)^{(g)},
%\end{equation*}
%and we denote $W_0^{(g)}=\{\calQ\in W_0: \deg(\calQ)=g\},$ then for each $g\in G(\calA(T)),$ we have that $W_0^{(g)}$ is a basis of the subspace $\calA(T)^{(g)}$ of $\calA(T).$ (see appendix \ref{ssec-nil-graded-alg} for notation).
% When it would be necessary we will denote $W_0^{(g)}(T)$ to avoid any confusion.

\subsection{The category $\bcalNGAT.$}\label{sec-the-category-NT}
In this section we define a category $\bcalNGAT.$ For the rest of the article, we consider $T=[t_{ij}]_{n\times n}$ and $S=[s_{ij}]_{m\times m}$ as strictly lower triangular matrices.
 We also write $\mathcal{X}=\{X_i:i=1,\dots, n\}$and $\mathcal{Y}=\{Y_i:i=1,\dots,m\}$ the corresponding set of generators of the algebras $\calA(T)$ and $\calA(S)$ respectively (as in definition \ref{def-nilalgebra-associated-to-T}).

%\subsubsection{Objects and morphisms of $\bcalNGAT$}\label{ssec-objects-of-NT}
 The objects of the category $\bcalNGAT$ are the nil graded algebras $\calA(T)$ strongly associated to triangular matrices. We also consider as an object of $\bcalNGAT$ the field $\F$ itself, considered as a $1-$dimensional $\F-$graded algebra, graduated by $\deg(z)=0$ for all $z\in \F.$

%\subsubsection{Morphisms of $\bcalNGAT$}\label{ssec-morphisms-of-NT}
The morphisms (resp. isomorphisms) of $\bcalNGAT$ are all the \emph{preserving degree} algebra homomorphisms (resp. isomorphisms) $\calA(T)\rightarrow\calA(S),$ ($T,S$ strictly lower triangular matrices).

Note that the only possible morphism $\F\rightarrow \calA(S),$ is the \emph{natural injection} of the field $\F$ into $\calA(S).$ Analogously the only possible morphism $\calA(T)\rightarrow\F$ is given by the assignment: $1\mapsto 1,\quad X_j\mapsto 0.$

More generally, given two objects $\calA(T),\calA(S)\neq \F$ of $\bcalNGAT,$ we always have a \emph{trivial} morphism  $\calA(T)\rightarrow\calA(S),$ given by the assignment, $1\mapsto 1$ and $X_j\mapsto 0.$
%Let us denote this morphism as  $\calO=\calO_{T,S}.$

In general the existence of nontrivial morphism between two arbitrary objects  $\calA(T),\calA(S)$ of $\bcalNGAT$ depends on some complicated conditions:
%For instance we develop some useful tools to deal with those complications:

  Note that any  morphism $\bgamma:\calA(T)\rightarrow\calA(S)$ is totally determined by the images of the set of generators: $X_r\mapsto\bgamma(X_r).$ Conversely given an assignment $X_r\mapsto\bgamma(X_r)$ that satisfies \begin{equation}\label{eq1-lemma-linear-transf-associated-to-morphism}
    (\bgamma(X_r))^2=\sum_{j<n}t_{rj}\bgamma(X_j)\bgamma(X_r),\quad r=1,\dots,n
  \end{equation}
  then, it defines a morphism $\bgamma:\calA(T)\rightarrow\calA(S)$ in the category $\bcalNT.$

 If $\bgamma:\calA(T)\rightarrow\calA(S)$ is a morphism in the category $\bcalNGAT.$ Then $\bgamma$ defines a matrix $\Gamma=[\gamma_{jr}]_{m\times n}$ by the equation:
   \begin{equation}\label{eq-matrix-Gamma-associated-to-morphism}
     \bgamma(X_r)=\sum_{j=1}^{m}\gamma_{jr}Y_j,\quad\textrm{for each}\quad r=1,\dots,n.
   \end{equation}

 The matrix $\Gamma$ will be called the \emph{matrix associated to the  morphism} $\bgamma.$

Conditions for the existence of isomorphisms between two objects of $\bcalNT$ are more complicated, but if we assume that we already have a morphism $\bgamma:\calA(T)\rightarrow\calA(S),$ then we can easily check if it is a isomorphism or not. In lemma \ref{lemma-conditions1-for-isomorphisms} and  corollary \ref{coro-conditions3-for-isomorphisms}, we consider $T=[t_{ij}]_{n\times n}$ and $S=[s_{ij}]_{n\times n}$ as strictly lower triangular matrices of the same size $n\times n.$

%The following lemmas provide us certain criteria to identify when a given morphism $\bgamma:\calA(T)\rightarrow\calA(S)$ in the category $\bcalNGAT,$ is in fact an isomorphism.

\begin{lemma}\label{lemma-conditions1-for-isomorphisms}
  Let $\bgamma:\calA(T)\rightarrow\calA(S)$ a morphism in the category $\bcalNGAT.$ Then $\bgamma$ is an  isomorphism in the category $\bcalNGAT$ if and only if $\bgamma(X_1)\bgamma(X_2)\cdots\bgamma(X_n)\neq 0.$
 % \begin{equation*}
 %   \bgamma(X_1)\bgamma(X_2)\cdots\bgamma(X_n)\neq 0.
 % \end{equation*}
\end{lemma}

\begin{proof}
 The image $\bgamma(\calA(T))\subset \calA(S)$ is a commutative subalgebra weakly associated to the matrix $T.$ In particular $\dim(\bgamma(\calA(T))\leq \dim(\calA(S))=2^n.$ Then $\bgamma$ is an isomorphism if and only if $\dim(\bgamma(\calA(T)))=2^n,$ but this is equivalent to say that $\bgamma(\calA(T))$ is strongly associated to $T.$ The conclusion follows from lemma \ref{lemma-monomial-basis}.
\end{proof}

%\begin{corollary}\label{coro-conditions2-for-isomorphisms}
 % Let $\bgamma:\calA(T)\rightarrow\calA(S)$ a morphism in the category $\bcalNGAT.$ Then $\bgamma$ is an  isomorphism in the category $\bcalNGAT$ if and only if the set $\{\bgamma(X_1),\bgamma(X_2),\cdots,\bgamma(X_n)\}$ is linearly independent in $\calA(S).$
%\end{corollary}
%\begin{proof}
%  Trivial
%\end{proof}
\begin{corollary}\label{coro-conditions3-for-isomorphisms}
  Let $\bgamma:\calA(T)\rightarrow\calA(S)$ be a morphism in the category $\bcalNGAT$ and let $\Gamma$ be the matrix associated to $\bgamma$. Then $\bgamma$ is an  isomorphism in the category $\bcalNGAT$ if and only if the matrix $\Gamma$ is invertible.
\end{corollary}
\begin{proof}
  If $\bgamma:\calA(T)\rightarrow\calA(S)$ is a morphism, then equation \ref{eq-matrix-Gamma-associated-to-morphism} in particular defines a linear application between the vector spaces generated by $\{X_j:j=1,\dots,n\}$ and $\{Y_j:j=1,\dots,n\},$ respectively. Therefore $\bgamma$ is an isomorphism  if and only if the matrix $\Gamma$ is invertible.
\end{proof}

%%%%%%%%%%%%%%%%%%%%%%%%%%%%%%%%%%%%%%%%%%%%%%%%%%%%%%%%%%%%%%

\begin{example}\label{ex-AT-AS-isomorphic}
  Consider the matrices:
  \begin{equation*}
    T=\left[\begin{matrix}
              0 & 0 & 0 & 0 \\
              0 & 0 & 0 & 0 \\
              -1 & -1 & 0 & 0 \\
              -1 & +1 & -1 & 0
            \end{matrix}\right],\quad
            S=\left[\begin{matrix}
               0& 0 & 0 & 0 \\
              -1& 0 & 0 & 0 \\
              -1&-1 & 0 & 0 \\
               0& +1& -1& 0
            \end{matrix}\right]
  \end{equation*}
  Let
  \begin{equation*}
    \mathcal{X}=\{X_1,X_2,X_3,X_4\},\quad \mathcal{Y}=\{Y_1,Y_2,Y_3,Y_4\}
  \end{equation*}
  the corresponding set of generators of the algebras $\calA(T)$ and $\calA(S)$ respectively (as in definition \ref{def-nilalgebra-associated-to-T}). Then the corresponding presentations are:
  \begin{equation*}
   \calA(T):\left\{\begin{array}{cc}
             X_1^2= 0&\quad \\
             \quad&\quad \\
             X_2^2= 0&\quad \\
              \quad&\quad\\
             X_3^2=-X_1X_3-X_2X_3&\quad \\
              \quad&\quad \\
             X_4^2= -X_1X_4+X_2X_4-X_3X_4&\quad
           \end{array}\right.
  \quad \textrm{and}\quad  \calA(S):
   \left\{\begin{array}{cc}
             Y_1^2= 0&\quad \\
             \quad&\quad \\
             Y_2^2= -Y_1Y_2&\quad \\
              \quad&\quad\\
             Y_3^2=-Y_1Y_3-Y_2Y_3&\quad \\
              \quad&\quad \\
             Y_4^2= \quad+Y_2Y_4-Y_3Y_4&\quad
           \end{array}\right.
  \end{equation*}
  We claim that $\calA(T)$ and $\calA(S)$ are isomorphic objects of the category $\bcalNGAT.$  In fact, if we consider the linear transformation $\bgamma$ induced by the assignment:
  \begin{equation*}
    \left\{\begin{array}{c}
             X_1\mapsto \bgamma(X_1)=Y_1 \\
             \quad\\
             X_2\mapsto \bgamma(X_2)=Y_1+2Y_2 \\
             \quad \\
             X_3\mapsto \bgamma(X_3)=2Y_3 \\
             \quad \\
             X_4\mapsto \bgamma(X_4)=2Y_4
           \end{array}\right.
  \end{equation*}

  we can check that:
  \begin{enumerate}
    \item
     \begin{equation*}
    \bgamma(X_1)^2=Y_1^2=0,
  \end{equation*}
  \item
  \begin{equation*}
    \bgamma(X_2)^2=(Y_1+2Y_2)^2=Y_1^2+4Y_1Y_2+4Y_2^2=+4Y_1Y_2-4Y_1Y_2=0,
  \end{equation*}
  \item
  \begin{equation*}
    \bgamma(X_3)^2=(2Y_3)^2=-4Y_1Y_3-4Y_2Y_3=-Y_1(2Y_3)-(Y_1+2Y_2)(2Y_3),
  \end{equation*}
  then
  \begin{equation*}
    \bgamma(X_3)^2=-\bgamma(X_1)\bgamma(X_3)-\bgamma(X_2)\bgamma(X_3),
  \end{equation*}
  \item
     \begin{equation*}
    \bgamma(X_4)^2=(2Y_4)^2=4Y_2Y_4-4Y_3Y_4=-(Y_1)(2Y_4)+(Y_1+2Y_2)(2Y_4)-(2Y_3)(2Y_4),
  \end{equation*}
  then
  \begin{equation*}
    \bgamma(X_4)^2=-\bgamma(X_1)\bgamma(X_4)+\bgamma(X_2)\bgamma(X_4)-\bgamma(X_3)\bgamma(X_4).
  \end{equation*}
  \end{enumerate}

  Therefore $\bgamma$ defines a morphism $\calA(T)\rightarrow \calA(S).$
 The matrix associated to $\bgamma$ is
  \begin{equation*}
    \Gamma=\left[ \begin{matrix}
             1 & 1 & 0 & 0 \\
             0 & 2 & 0 & 0 \\
             0 & 0 & 2 & 0 \\
             0 & 0 & 0 & 2
           \end{matrix}\right]
  \end{equation*}
  then $\det(\Gamma)=2^3\neq 0$ so $\Gamma$ is invertible, and $\bgamma$ is an isomorphism.
\end{example}

%The following example illustrates the difficulties associated to the existence of isomorphisms between two arbitrary objects:

%\begin{example}
%  Considering the following matrices:
%  \begin{equation*}
%    T=\left[\begin{matrix}
%              0 & 0 & 0 \\
%              0 & 0 & 0 \\
%              1 & 0 & 0
%            \end{matrix}\right],\quad S=\left[\begin{matrix}
%              0 & 0 & 0 \\
%              0 & 0 & 0 \\
%              1 & 1 & 0
%            \end{matrix}\right].
%  \end{equation*}
%  We claim that the algebras $\calA(T),\calA(S)$ are not isomorphic. We left the details as an exercise to the reader.
%\end{example}

%Note that in all the results given in this subsection, we are assuming that we start with a morphism between two objects (with the same dimension) in the category $\bcalNT$, and then we obtain certain criteria to identify if it is or not an isomorphism. The existence of non trivial morphism between two arbitrary objects $\calA(T),\calA(S)$ (with $T=[t_{ij}]_{n\times n},S=[s_{ij}]_{m\times m},$ strictly lower triangular matrices, not necessarily with the same size) is a more complicated problem, that we expect to developed in a future article. For instance we have the following lemma:

\subsection{The $\nabla-$products in $\bcalNGAT$}\label{ssec-nabla-prod-NT}

Given two strictly lower triangular matrices $T=[t_{ij}]_{n\times n}$ and $S=[s_{ij}]_{m\times m}$ and an arbitrary matrix $C=[c_{ij}]_{m\times n}$ we define the matrix $T\nabla_C S=[a_{ij}]_{(n+m)\times(n+m)}$ as follows:

\begin{equation*}
  a_{ij}=\left\{\begin{array}{cc}
                  t_{ij} & \quad\textrm{if}\quad 1\leq i,j \leq n \\
                  \quad & \quad \\
                  c_{kj} & \quad\textrm{if}\quad i=n+k\quad \textrm{and}\quad 1\leq j\leq n \\
                  \quad & \quad \\
                  s_{kc} & \quad \textrm{if}\quad i=n+k\quad \textrm{and}\quad j=n+c\\
                  \quad & \quad \\
                  0 & \quad \textrm{otherwise}
                \end{array}\right.
\end{equation*}

that is

\begin{equation*}
  T \nabla_C S=\left[\begin{matrix}
                       T & 0 \\
                       C & S
                     \end{matrix}\right].
\end{equation*}

By definition $T\nabla_C S$ is a strictly lower triangular matrix, then there is a nil graded algebra strongly associated to $T\nabla_C S$:

\begin{definition}\label{def-nabla-product-for-algebras}
  Given two strictly lower triangular matrices $T=[t_{ij}]_{n\times n}, S=[s_{ij}]_{m\times m}$ and a matrix $C=[c_{ij}]_{m\times n},$ we define the $\nabla-$product of $\calA(T)$ and $\calA(S)$ relative to $C,$  denoted by $\calA(T)\nabla_C \calA(S),$  as the nil graded algebra strongly associated to the matrix $T\nabla_C S.$ That is:
  \begin{equation*}
    \calA(T)\nabla_C \calA(S)=\calA(T\nabla_C S).
  \end{equation*}
\end{definition}

%\begin{remark}\label{remark-def-nabla-product-for-algebras}
By definition, the algebra $\calA(T)\nabla_C\calA(S)$ has a presentation, given by a set of generators $$\calZ=\{Z_1,\dots,Z_n,Z_{n+1},\dots,Z_{n+m}\}$$ and relations:

\begin{equation}\label{eq1-presentation-for-AT-nablaC-AS}
  Z_k^2=\sum_{j<k}{t_{kj}Z_jZ_k},  \quad\textrm{if}\quad k\leq n,
\end{equation}
and

\begin{equation}\label{eq2-presentation-for-AT-nablaC-AS}
  Z_{n+k}^2=\sum_{j\leq n}{c_{kj}Z_jZ_{n+k}}+\sum_{j\leq k }{s_{kj}Z_{n+j}Z_{n+k}},  \quad\textrm{if}\quad k\leq m .
\end{equation}
%\end{remark}

\begin{example}\label{ex1-nabla-product}
  Let
  \begin{equation*}
    T={\left[\begin{matrix}
              0 & 0 & 0 & 0 \\
              -1 & 0 & 0 & 0 \\
              -1& -1 & 0 & 0 \\
              -1 & -1 & -1 & 0
            \end{matrix}\right]},\quad
    S={\left[\begin{matrix}
              0 & 0 \\
              2 & 0
            \end{matrix}\right]}.
  \end{equation*}
  and take
  \begin{equation*}
    C={\left[\begin{matrix}
              1 & 2 & 0 & 0 \\
              0 & 0 & 1 & -1
            \end{matrix}\right]}.
  \end{equation*}
  %(we add colors to illustrate the idea of the construction)

  Then
  \begin{equation*}
    T\nabla_C S=\left[\begin{matrix}
                        {0} & {0}& {0}& {0} & 0 & 0 \\
                        {-1}& {0}& {0}& {0} & 0 & 0 \\
                        {-1}&{-1}& {0}& {0} & 0 & 0 \\
                        {-1}&{-1}&{-1}& {0} & 0 & 0 \\
                         {1}& {2}& {0}& {0} & {0} & {0} \\
                         {0}& {0}& {1}&{-1} & {2} & {0}
                      \end{matrix}\right]
  \end{equation*}
  Then the presentation of the algebra $\calA(T)\nabla_C \calA(S),$ is given by the set of generators $\{Z_1,Z_2,Z_3,Z,_4,Z_5,Z_6\}$ and relations:

  \begin{equation*}
    \left\{\begin{array}{cc}
             Z_1^2= 0&\quad \\
             \quad&\quad \\
             Z_2^2=-Z_1Z_2&\quad \\
              \quad&\quad\\
             Z_3^2=-Z_1Z_3-Z_2Z_3&\quad \\
              \quad&\quad \\
              Z_4^2=-Z_1Z_4-Z_2Z_4-Z_3Z_4&\quad \\
              \quad&\quad \\
              Z_5^2=Z_1Z_5+2Z_2Z_5&\quad \\
              \quad&\quad \\
             Z_6^2=Z_3Z_6-Z_4Z_6+2Z_5Z_6&\quad
           \end{array}\right.
  \end{equation*}
\end{example}

\begin{lemma}\label{lemma-incrustation-AT-into-ATnablaAS}
  There is a natural monomorphism $\calA(T)\rightarrow \calA(T)\nabla_C\calA(S)$ in the category $\bcalNGAT.$
\end{lemma}
\begin{proof}
  Let $\calX=\{X_1,\dots,X_n\}$ and $\calZ=\{Z_1,\dots,Z_{n+m}\}$ the corresponding set of generators of $\calA(T)$ and $\calA(T)\nabla_C\calA(S).$ Consider the assignment:
  \begin{equation*}
    \iota:X_i\mapsto Z_i,\quad i=1,\dots,n.
  \end{equation*}

  Let $\bcalU$ the algebra generated by the images of $\iota.$ Then by equation \ref{eq2-presentation-for-AT-nablaC-AS}, $\iota$ induces a morphism $\calA(T)\rightarrow \calA(T)\nabla_C\calA(S)$  and $\bcalU$ is a graded algebra weakly associated to the matrix $T.$

  Since the algebra $\calA(T)\nabla_C\calA(S)$ is, by definition (and theorem \ref{theo-all-AT-is-strongly}), strongly associated to the matrix $T\nabla_C S,$ we have in particular that the element $Z_1\cdots Z_n$ is different from zero, then by lemma \ref{lemma-monomial-basis} we have that $\dim(\bcalU)=\dim(\calA(T))$ and therefore $\iota: \calA(T)\rightarrow \calA(T)\nabla_C\calA(S)$ is a monomorphism.
\end{proof}

%Note that the restricted matrix $(T\nabla_CS)|_{n}$ is equal to $T,$ then the algebra $\iota(\calA(T))$ correspond to the subalgebra $\calA((T\nabla_CS)|_{n})$ of $\calA(T)\nabla_C \calA(S).$

%From now and so on, we identify the algebra $\calA(T)$ with its image $\iota(\calA(T))$ in $\calA(T)\nabla_C \calA(S).$

\begin{lemma}\label{lemma-quotient-ATnablaAS-over-AS}
  There is a natural epimorphism $\calA(T)\nabla_C\calA(S)\rightarrow\calA(S)$ in the category $\bcalNGAT.$
\end{lemma}
\begin{proof}
  Let $\calZ=\{Z_1,\dots,Z_{n+m}\}$ and $\calY=\{Y_1,\dots,Y_m\}$ the corresponding set of generators of $\calA(T)\nabla_{C}\calA(S)$ and $\calA(S)$ respectively. Consider the assignment:
  \begin{equation*}
    \pi: \left\{\begin{array}{cc}
                  Z_j\mapsto 0 & \quad\textrm{if}\quad 1\leq j\leq n \\
                  \quad & \quad \\
                  Z_{n+j}\mapsto Y_j & \quad\textrm{if}\quad 1\leq j\leq m
                \end{array}\right.
  \end{equation*}
    The images of $\pi,$ satisfy relations \ref{eq1-presentation-for-AT-nablaC-AS} and \ref{eq2-presentation-for-AT-nablaC-AS} trivially. Then $\pi$ defines a morphism $\pi:\calA(T)\nabla_C\calA(S)\rightarrow\calA(S).$

    Let $\bcalU\subset \calA(S)$ the algebra generated by the images of $\pi.$ Since $\bcalU$ contains all the generators of $\calA(S),$ we have $\bcalU=\calA(S),$ and $\pi$ is an epimorphism.
\end{proof}

The last lemma implies that $\calA(S)$ is isomorphic as graded algebra to the quotient graded algebra $$\left(\calA(T)\nabla_C\calA(S)\right)/\langle Z_1,\dots,Z_n\rangle,$$
where $\langle Z_1,\dots,Z_n\rangle$ denotes the ideal of $\calA(T)\nabla_C\calA(S)$ generated by the elements $Z_1,\dots,Z_n\in \calA(T)\nabla_C\calA(S).$

%\vspace
%%%%%%%%%%%%%%%%%%%%%%%%%%%
The following lemma provides us of some \emph{weak associative property} for the $\nabla-$products in $\bcalNT$:

\begin{lemma}\label{lemma-associativity-nabla-prod-matrix}
Given three strictly lower triangular matrices $T=[t_{ij}]_{n\times n}, S=[s_{ij}]_{m\times m},U=[u_{ij}]_{p\times p}$ and three arbitrary matrices $C=[c_{ij}]_{m\times n},D=[d_{ij}]_{p\times m},F=[f_{ij}]_{p\times n}$ then we have:

\begin{equation*}
  (\calA(T)\nabla_{C} \calA(S))\nabla_{\widetilde{D}}\calA(U)=\calA(T)\nabla_{\widetilde{C}}(\calA(S)\nabla_{D}\calA(U)),
\end{equation*}

where $\widetilde{C},\widetilde{D}$ are the adequate extensions of $C,D$ by $F$ defined as:
\begin{equation*}
  \widetilde{C}=\left[\begin{array}{c}
                        C  \\
                        F
                      \end{array}\right]_{(m+p)\times n },\quad \widetilde{D}=\left[\begin{array}{cc}
                                                                     F & D
                                                                   \end{array}\right]_{p\times (m+n) }.
\end{equation*}

\end{lemma}

\begin{proof}
  We only have to prove that
  \begin{equation*}
  (T\nabla_{C} S)\nabla_{\widetilde{D}}U=T\nabla_{\widetilde{C}}(S\nabla_{D}U),
\end{equation*}

  By definition:

  \begin{equation*}
   (T\nabla_{C} S)\nabla_{\widetilde{D}}U=\left[\begin{array}{cc}
                                                  T\nabla_{C} S & 0 \\
                                                  \widetilde{D} & U
                                                \end{array}\right] =\left[\begin{array}{ccc}
                                                                            T & 0 & 0 \\
                                                                            C & S & 0 \\
                                                                            F & D & U
                                                                          \end{array}\right]
  \end{equation*}
  and

  \begin{equation*}
   T\nabla_{\widetilde{C}}(S\nabla_{D}U)=\left[\begin{array}{cc}
                                                  T & 0 \\
                                                  \widetilde{C} & S\nabla_{D}U
                                                \end{array}\right] =\left[\begin{array}{ccc}
                                                                            T & 0 & 0 \\
                                                                            C & S & 0 \\
                                                                            F & D & U
                                                                          \end{array}\right]
  \end{equation*}

  and the desired result follows.
\end{proof}

%\medskip

%Then lemma \ref{lemma-associativity-nabla-prod-matrix} implies the following

%\begin{corollary}
%  The category $\bcalNT$ is a monoidal category under $\nabla_0$ product.
%\end{corollary}

We denote by $\calA(T)\nabla_0\calA(S)$ whenever $C$ is a zero matrix. We have directly from lemma \ref{lemma-associativity-nabla-prod-matrix} that:

\begin{corollary}
  Given three strictly lower triangular matrices $T,S,U,$ then we have
  \begin{equation*}
    (\calA(T)\nabla_0\calA(S))\nabla_0\calA(U)=\calA(T)\nabla_0(\calA(S)\nabla_0\calA(U)).
  \end{equation*}
\end{corollary}

 The following lemma provides us with a \emph{commutative property} for the products $\nabla_0$ in the category $\bcalNGAT:$

\begin{lemma}\label{lemma-AT-tensor-AS-is-commutative}
The objects $ \calA(T)\nabla_0 \calA(S)$ and $\calA(S)\nabla_0 \calA(T)$ are isomorphic in the category $\bcalNGAT$.
\end{lemma}
\begin{proof}
  Let $\{Z_1,\dots,Z_{n+m}\}$ and $\{V_1,\dots,V_{n+m}\}$ the set of generators of $ \calA(T)\nabla_0 \calA(S)$ and $\calA(S)\nabla_0 \calA(T),$ respectively.
  Then we have
  \begin{equation*}
    Z_{r}^{2}=\left\{\begin{array}{cc}
             \sum_{j<r}{t_{jr}}Z_jZ_r, & \textrm{if}\quad 1\leq r\leq n \\
             \quad & \quad \\
              \sum_{n<j<r}{s_{{j-n},{r-n}}}Z_jZ_r,& \textrm{if}\quad n< r \leq n+m
           \end{array}\right.
  \end{equation*}
  and
  \begin{equation*}
    V_{r}^{2}=\left\{\begin{array}{cc}
             \sum_{j<r}{s_{jr}}V_jV_r, & \textrm{if}\quad 1\leq r\leq m \\
             \quad & \quad \\
              \sum_{m<j<r}{t_{{j-m},{r-m}}}V_jV_r,& \textrm{if}\quad m< r \leq n+m
           \end{array}\right.
  \end{equation*}
  Therefore the assignment
  \begin{equation*}
    \bgamma:\left\{\begin{array}{cc}
                     Z_r\mapsto V_{r+m} & \quad \textrm{if}\quad r=1,\dots,n \\
                     Z_r \mapsto V_{r-n}& \quad \textrm{if}\quad r=n+1,\dots,n+m
                   \end{array}\right.
  \end{equation*}
  defines a morphism. Since the product $\bgamma(Z_r)\cdots \bgamma(Z_{n+m})\neq 0,$ we conclude that $\bgamma$ is an isomorphism.
\end{proof}

\section{Applications to Soergel Calculus.}\label{sec-Applic-Soergel}

\subsection{The Diagrammatic Soergel category }\label{ssec-Soergel-category}
%\subsection{Definition}
Following the work of Elias and Williamson (\cite{EW}), given a Coxeter system of finite rank $(W,S)$ (relative to the set of indexes $J$ and the Coxeter matrix $\mathbf{m}$), we consider a \emph{realization} $(\mathfrak{h},\{\alpha^{\vee}_a\},\{\alpha_a\})$ of $(W,S)$ over the field $\F,$ that is a $\F-$vector space $\mathfrak{h},$ and subsets $\{\alpha^{\vee}_a: a\in J\}\subset\mathfrak{h},$ and  $\{\alpha_a:a\in J\}\subset\mathfrak{h}^{\ast},$  where the equations:

%\begin{equation}\label{definition-action-root-coroot}
%  \langle \alpha_a,\alpha^{\vee}_b \rangle=\left\{ \begin{array}{cc}
%                                                     -2\cos(\pi/{\mathbf{m}(a,b)}) & \quad\textrm{if}\quad \mathbf{m}(a,b)<\infty \\
%                                                     -2 & \quad\textrm{if}\quad \mathbf{m}(a,b)=\infty
%                                                   \end{array}\right.
%\end{equation}

%The \emph{Cartan Matrix} of the realization is defined as $\mathcal{C}=[\langle \alpha_a,\alpha^{\vee}_b \rangle]$

\begin{equation}\label{definition-geometric-representation}
  \left\{\begin{array}{cc}
           \beta\cdot s_{a} =\beta-\langle\alpha_a,\beta \rangle\alpha^{\vee}_a, & \quad (\beta \in \mathfrak{h}). \\
           \quad & \quad \\
           \gamma\cdot s_{a}=\gamma-\langle\gamma,\alpha^{\vee}_a\rangle \alpha_a, & \quad (\gamma \in \mathfrak{h}^{\ast}).
         \end{array}\right.
 % \beta\cdot s_{a} =\beta-\langle\alpha_a,\beta \rangle\alpha^{\vee}_a,\quad (\beta \in \mathfrak{h}).
\end{equation}
%\begin{equation}\label{definition-dual-geometric-representation}
%  \gamma\cdot s_{a}=\gamma-\langle\gamma,\alpha^{\vee}_a\rangle \alpha_a,\quad (\gamma \in \mathfrak{h}^{\ast}).
%\end{equation}

define \emph{representations} of the group $W$ on $\mathfrak{h}$ and $\mathfrak{h}^{\ast}$ respectively.

%In particular, we have

Let $\bcalR=S(\mathfrak{h}^{\ast})$ be the symmetric algebra of $\mathfrak{h}^{\ast}$ over $\F,$ and consider the grading in $\bcalR$ induced by the equation $\deg(\mathfrak{h}^{\ast})=2.$ The action of $W$ on $\mathfrak{h}^{\ast}$ extends to an action on $\bcalR.$

For each $a\in J$ we define the \emph{Demazure operator} $\partial_a:\bcalR\rightarrow \bcalR(-2)$ (see \cite{EW} for details) as follows:
\begin{equation}\label{def-demazure-operator}
  \partial_a (f)=\dfrac{f- f\cdot s_a}{\alpha_a}.
\end{equation}

Note that for $\gamma\in \mathfrak{h}^{\ast}\subset \bcalR$ we have $\partial_a(\gamma)=\langle \gamma,\alpha^{\vee}_a\rangle.$

In the following we refer to the element of $J$ as \emph{colors}. Sometimes will be useful to add a coloration to differentiate element on $J,$ for example we write ${\color{blue}a}$ and ${\color{red}b},$ instead of $a$ and $b.$

\begin{definition}\label{def-soergel-graph}
  A \emph{Soergel Graph} relative to the Coxeter system $(W,S)$ is a finite and decorated graph $\calS$ embedded in the planar strip $\mathbb{R}\times[0,1].$ The edges of a Soergel graph are \emph{colored} by the colors $a\in J.$ The vertices in this graphs are of the following types: Univalent Vertices ($V1$) and Trivalent Vertices ($V3$) involving one color, and $2m-$\emph{valent vertices} ($V2m$), involving pair of colors ${\color{blue}a},{\color{red}b}\in J$ such that $m=\mathbf{m}({\color{blue}a},{\color{red}b})<\infty.$
  \begin{equation*}\label{draw-univalent-vertex}
  \begin{tikzpicture}[xscale=0.5,yscale=0.5]
\draw[thick,blue] (0-1,0)--(0-1,1);
\node[below] at (0.5-1,0.5) {;};
\draw[fill,blue] (0-1,1) circle [radius=0.1];
\node[] at (0-1,-1) {$V1$};
%\end{tikzpicture}\quad
%\begin{tikzpicture}[xscale=0.5,yscale=0.5]
\draw[thick,blue] (2,2)--(0+2,1);
\draw[thick,blue] (2,1)--(-1+2,0);
\draw[thick,blue] (2,1)--(1+2,0);
\node[below] at (3.5,0.5) {;};
\node[] at (2,-1) {$V3$};
%\node[below] at (0,-0.1) {$a$};
%\draw[fill] (0,1) circle [radius=0.1];
%\node[] at (5,0.5) {$({\color{blue}a}\in J.)$};
%\end{tikzpicture}\quad
% \begin{tikzpicture}[xscale=0.5,yscale=0.5]
\draw[thick,blue] (2+4,2)--(0+4,0);
\draw[thick,green] (0+4,2)--(2+4,0);
\node[below] at (6.5,0.5) {;};
\node[] at (5,-1) {$V4$};
%\draw[thick,blue] (0,1)--(1,0);
%\node[below] at (0,-0.1) {$a$};
%\draw[fill] (0,1) circle [radius=0.1];
%\node[] at (0+4,-1) {$(\mathbf{m}({{\color{blue}a},{\color{green}c}})=2).$};
\draw[thick,blue] (0+8,2)--(0+8,1);
\draw[thick,blue] (0+8,1)--(-1+8,0);
\draw[thick,blue] (0+8,1)--(1+8,0);
\draw[thick,red] (0+8,0)--(0+8,1);
\draw[thick,red] (0+8,1)--(-1+8,2);
\draw[thick,red] (0+8,1)--(1+8,2);
\node[] at (8,-1) {$V6$};
\node[below] at (9.5,0.5) {;};
\node[] at (11,0) {etc.};
\end{tikzpicture}
\end{equation*}

(Here we are assuming that $\mathbf{m}({{\color{blue}a},{\color{green}c}})=2$ and $\mathbf{m}({{\color{blue}a},{\color{red}b}})=3$)

  The different connected components of the complement of $\calS$ in $\mathbb{R}\times[0,1]$ are called the \emph{regions} of $\calS.$
   Each region of $\calS$ may be decorated with one or several patches, each of them labeled by homogeneous element $f\in \bcalR.$
\end{definition}

%The point of intersection of an edge and the boundary of the strip $\mathbb{R}\times[0,1]$ is called a \emph{boundary point}. Boundary points are not considered as vertices of the graph.

% (We will use the same color code in many of the following examples)

\begin{figure}
  \centering
  \includegraphics[scale=0.25]{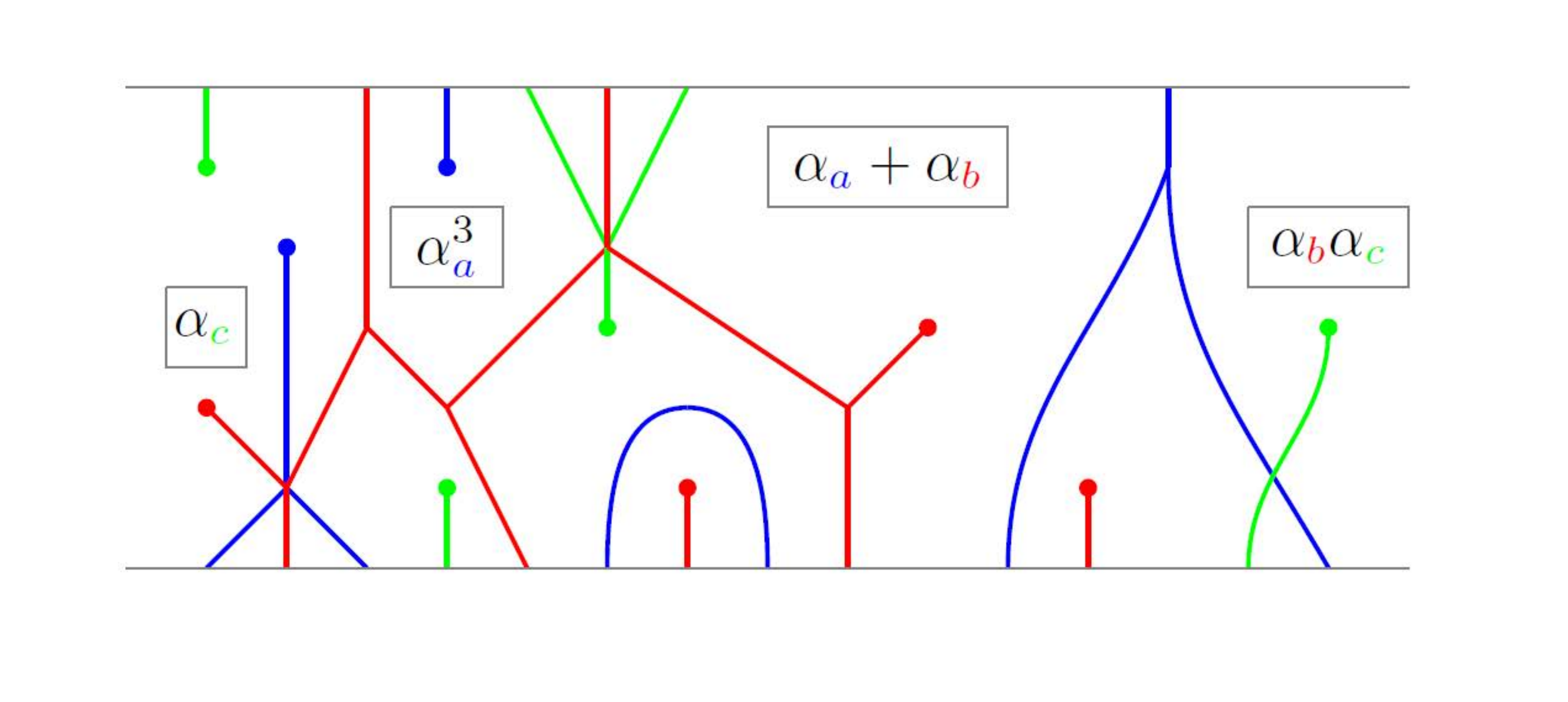}
  \caption{A Soergel graph}\label{FIG-A1A2}
\end{figure}

%\includegraphics[scale=0.2]{FIG-A2.pdf}

%A Soergel graph have at most two not bounded regions, the leftmost and the rightmost. If $\calS$ is a Soergel graph, $f,g\in \bcalR$ are homogeneous elements, then we write $f\calS$ (resp $\calS g$) when we add a patch decorated by $f$ in the leftmost region (resp. in the rightmost) of $\calS.$

% If we consider the $\F-$vector space $\bcalV,$ generated by all the Soergel Graphs, then the last notation defines (by linearity) a two-sided action of the algebra $\bcalR$ on $\bcalV,$ endowing $\bcalV$ with an structure of $\bcalR-$bimodule.

The \emph{degree} of a Soergel graph $\calS$ is defined as the sum of the degrees of each of the homogeneous elements $f\in \bcalR$ that appear in the patches of $\calS,$ plus the sum of the degree of each vertex, where by definition, the univalent vertices have degree $+1,$ trivalent vertices have degree $-1$ and the other vertices have degree $0.$ The boundary points of a Soergel graph define two sequences of colors, the \emph{top boundary} and the \emph{bottom boundary}. We naturally identify each sequence of colors  $(a_1,\dots,a_p)\in J^{p}$ with an expression $\underline{w}=s_{a_1}s_{a_2}\cdots s_{a_p}\in W^{\ast}.$
%We  also associate to the identity element of the monoid $W^{\ast}$ the \emph{empty sequence} $\emptyset.$

%Given a Soergel graph $\calS,$ we denote $\calS^{a}$ the Soergel graph obtained from $\calS$ by applying a vertical flip.
%If $\calS$ is the Soergel graph in the left of figure \ref{FIG-A1A2}, then $\calS^{a}$ is the Soergel graph in the right of figure \ref{FIG-A1A2}.

%We will concentrate in Soergel graphs, whose top and bottom boundaries are associated with reduced expressions $\underline{u},\underline{v}\in W^{\ast}_m.$ One can easily check that this is the case in the graph in \ref{draw-Soergel-graph}.

\begin{definition}\label{def-Soergel-Category}
  The \emph{diagrammatic Soergel Category} relative to the Coxeter system $(W,S),$ is defined as the monoidal category $\calD,$ whose objects are all the \emph{color sequences} (including the empty one) and whose morphisms sets $\textrm{Hom}_\calD(\underline{u},\underline{v})$ are the $\F-$vector spaces generated by all the Soergel graphs with bottom boundary $\underline{u}$ and top boundary $\underline{v},$ modulo isotopy and modulo a series of local relations, that we partially describe  in the following list (we do not use the complete list of relations in this article, to see the complete definition, read \cite{EW}):
\begin{itemize}
\item \begin{equation}\label{eq-polynomial-relation-1}
\begin{tikzpicture}[xscale=0.5,yscale=0.5]
\draw[thick,blue] (0,0)--(0,1);
\node[] at (3.2,0.5) {$=\quad\alpha_{\color{blue}a}$};
\draw[thin,gray] (4.5,1)--(3.5,1);
\draw[thin,gray] (4.5,0)--(3.5,0);
\draw[thin,gray] (4.5,0)--(4.5,1);
\draw[thin,gray] (3.5,0)--(3.5,1);
\draw[dashed,gray] (0,0.5) circle [radius=1];
\draw[dashed,gray] (4,0.5) circle [radius=1];

\draw[fill,blue] (0,1) circle [radius=0.1];
\draw[fill,blue] (0,0) circle [radius=0.1];
\end{tikzpicture}
%; \includegraphics[scale=0.17]{FIG-A3.pdf}
%\quad \includegraphics[scale=0.15]{FIG-A4.pdf};
%\quad\includegraphics[scale=0.1]{FIG-A5.pdf}%
\end{equation}
\item \begin{equation}\label{eq-polynomial-relation-2}
      \includegraphics[scale=0.17]{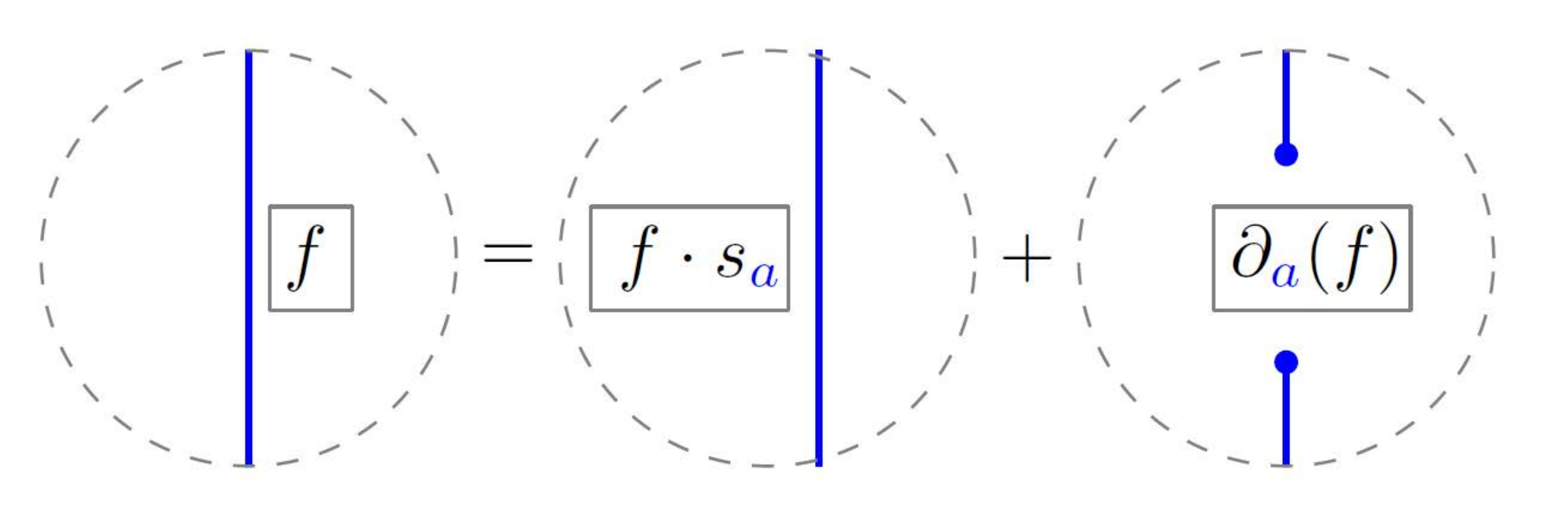}
\end{equation}
%\end{itemize}
 %\item \begin{equation}\label{eq-one-color-relation-1}
 % \includegraphics[scale=0.15]{FIG-A4.pdf},\quad\includegraphics[scale=0.1]{FIG-A5.pdf}% \includegraphics[scale=0.13]{FIG-A6.pdf}
%\end{equation}
\item If $f\in \bcalR$ is an homogeneous element and $\calS$ is a Soergel graph, then
      \begin{equation}\label{eq-left-annihilation}
             f\calS =0,\quad\textrm{whenever}\quad \deg(f)>0.
      \end{equation}
 \item \emph{Other relations} (see \cite{EW}).
  \end{itemize}%\end{itemize}
%\end{enumerate}

  If $\calS$ is a Soergel graph in $\textrm{Hom}_{\calD}(\underline{u},\underline{v}),$ and $\calT$ is a Soergel graph in $\textrm{Hom}_{\calD}(\underline{v},\underline{w}),$ then the composition $\calT\calS$ is defined by vertical concatenation ($\calT$ above of $\calS$). The composition
  $\textrm{Hom}_{\calD}(\underline{v},\underline{w})\times \textrm{Hom}_{\calD}(\underline{u},\underline{v})\rightarrow\textrm{Hom}_{\calD}(\underline{u},\underline{w}),$
  is defined by linearity. The monoidal operation in $\calD$ is defined by horizontal concatenation and denoted by $\otimes.$
\end{definition}

%Each of the  spaces $\textrm{Hom}_{\calD}(\underline{u},\underline{v})$ in this category, can be seen as $\bcalR-$bimodules.

%Relation \ref{eq-cyclotomic-relation} is not in the original definition given by Elias and Williamson, so this is an adaptation of the \emph{original} Diagrammatic Soergel category. By this relation, we have that $R$ acts trivially by the left on each $\textrm{Hom}_{\calD}(\underline{u},\underline{v}).$  For example by relation \ref{eq-cyclotomic-relation}, we have that the graph in \ref{draw-Soergel-graph} is equal to zero.

%By definition the group $W$ act on $\bcalR,$ then in particular it act on scalars $r\in \F.$ In that case we have $r \cdot s_a =r,$ and therefore $\partial_a(r)=0.$ That fact combined with relation \ref{eq-polynomial-relation-2}, imply that if $r\in \F$ is decorating any region of a graph $\calS,$ then $r$ can be moved freely (not affecting the other elements of $\calS$) to the leftmost (resp rightmost) region of $\calS.$ Therefore, we can understand the scalar multiplication $r\calS=\calS r,$ in the $\bcalR-$bimodule $\textrm{Hom}_{\calD}(\underline{u},\underline{v})$  as the action of adding $r$ as a decoration in any region of $\calS.$

%\subsubsection{Special Soergel graphs}\label{ssec-special-Soergel-graphs}

%\subsection{Libedinsky's double light leaves}\label{ssec-Double-light-leaves}

Thanks to the work of Libedinsky, we know that there are explicit bases
(called \emph{the double light leaves bases}) $\mathbb{L}\mathbb{L}_{\ulu}^{\ulv}(\calS,\calR)$ for the many vector spaces $\textrm{Hom}_{\calD}(\underline{u},\underline{v}).$ We will mention some of the elements of those bases later. For details read \cite{LibLightLeaves}.

%The following theorem is due to Libedinsky (\cite{LibLightLeaves})
%\begin{theorem}\label{theo-double-leaves-is-basis}
%Given two expressions $\ulu,\ulv\in W^{\ast},$ then the set of all double leaves $\mathbb{L}\mathbb{L}_{\ulu}^{\ulv}(\calS,\calR),$  is a basis of the right (resp. left) $\bcalR-$module $\textrm{Hom}_{\calD}(\ulu,\ulv).$
%\end{theorem}

%\subsection{Left annihilation}

\subsection{Nil graded algebras associated to free expressions}\label{ssec-nil-alg-and-expressions}

Let $(W,S)$ a Coxeter system. We start this subsection by defining an action of the free monoid $W^{\ast}$ on the algebra $\bcalR.$ This action is an adapted version of the action of the group $W$ on $\bcalR.$

\begin{definition}
Let $\ulw\in W^{\ast}$ an expression of $w\in W,$  then we define:
\begin{equation}\label{eq-def-monoid-acting-on-bcalR}
  \underline{w}\circ f=f\cdot w^{-1},\quad (f\in \bcalR).
\end{equation}
\end{definition}

The right side of equation \ref{eq-def-monoid-acting-on-bcalR} corresponds to the action of the group $W$ acting on $\bcalR.$ Since the action of the group is well defined, the action of the monoid $W^{\ast}$ is well defined (even if we take a not reduced expression of $w$ in equation \ref{eq-def-monoid-acting-on-bcalR}).

More particularly we have:

\begin{equation}\label{eq-def-monoid-acting-on-roots}
  s_{a}\circ f=f\cdot s_{a},\quad (f\in \bcalR).
\end{equation}
where the $s_a$ represent an expression in $W^{\ast}$ at the left and an element of the group $W$ at the right side of equation \ref{eq-def-monoid-acting-on-roots}.

Also note that for two expressions $\underline{u},\underline{v}\in W^{\ast}$ we have

\begin{equation}\label{eq-def-monoid-acting-on-roots-2}
  (\underline{u}\underline{v})\circ f=\underline{u}\circ (\underline{v}\circ f).
\end{equation}

Given a column vector $Q=[q_{k1}]_{g\times 1}$ in   $\bcalR^{g}=\bcalR\times\cdots\times\bcalR,$ we define the action of $\ulw$ on $Q$ as
  \begin{equation*}
    \ulw\circ Q=\left[\begin{matrix}
                        \ulw \circ q_{11} \\
                        \ulw \circ q_{21} \\
                        \vdots \\
                        \ulw \circ q_{g1}
                      \end{matrix}\right]
  \end{equation*}

The following definition extends the Demazure operator.
%This new concepts are very important for our purposes:

\begin{definition}\label{def-adapted-Demazure-1}
  Let $\underline{w}=s_{a_1}s_{a_2}\cdots s_{a_p}\in W^{\ast},$ for $r=1,\dots,p-1$ let $\ulw_{>r}=s_{a_{r+1}}\cdots s_{a_p}.$ We define the \emph{generalized Demazure operator} $\bpartial_{\underline{w}},$ as follows:

\begin{enumerate}
\item  If $f\in \bcalR$ we write:
%\begin{enumerate}
%\item We first define $\bpartial_{a_p}$ by

%\begin{equation*}
%  \bpartial_{a_p}(f)=[\partial_{a_p}(f)],\quad f\in \bcalR.
%\end{equation*}
%where $[\partial_{a_p}(f)]$ is a $1\times 1$ matrix.
%\item If we already have defined the operator  $\bpartial_{\ulw_{>k}}$
%then we define $\bpartial_{(s_{a_k}\ulw_{>k})}$ as follows:
%\begin{equation*}
%  \bpartial_{(s_{a_k}\ulw_{>k})}(f)=[\partial_{a_k}(\ulw_{>k}\circ f),\bpartial_{\ulw_{>k}}(f)],\quad (f\in \bcalR).
%\end{equation*}
%Then we have:
\begin{equation*}
  \bpartial_{\ulw}(f)=\left[\begin{array}{ccccc}
                              \partial_{a_1}(\ulw_{>1}\circ f), & \partial_{a_2}(\ulw_{>2}\circ f), & \partial_{a_3}(\ulw_{>3}\circ f), & \cdots
                              & ,\partial_{a_{p}}(f)
                            \end{array}\right]
\end{equation*}
%\end{enumerate}
\item Given a column vector $Q=[q_{k1}]_{g\times 1}$ in   $\bcalR^{g}=\bcalR\times\cdots\times\bcalR,$ we write:

\begin{equation*}
  \bpartial_{\ulw}(Q)=\left[\begin{matrix}
                               \bpartial_{\ulw}(q_{11}) \\
                               \bpartial_{\ulw}(q_{21}) \\
                               \vdots\\
                               \bpartial_{\ulw}(q_{g1})
                             \end{matrix}\right]=
  \left[\begin{array}{cccc}
                              \partial_{a_1}(\ulw_{>1}\circ q_{11}), & \partial_{a_2}(\ulw_{>2}\circ q_{11}), & \cdots
                              &,\partial_{a_p}(q_{11}) \\
                              \partial_{a_1}(\ulw_{>1}\circ q_{21}), & \partial_{a_2}(\ulw_{>2}\circ q_{21}), &\cdots &,\partial_{a_p}(q_{21})\\
                              \vdots & \vdots & \vdots & \vdots\\
                             \partial_{a_1}(\ulw_{>1}\circ q_{g1}), & \partial_{a_2}(\ulw_{>2}\circ q_{g1}), & \cdots &,\partial_{a_p}(q_{g1}) \\
                          \end{array}\right].
\end{equation*}
\end{enumerate}
%\end{definition}

%Note that equation \ref{eq-def-adapted-Demazure-2} implies that $\bpartial_{\ulw}(f)$ is a row vector in $\bcalR(-2)^{p}.$
\end{definition}

%Note that equation \ref{eq-def-adapted-Demazure-2} implies that $\bpartial_{\ulw}(f)$ is a row vector in $\bcalR(-2)^{p}.$
Note that when we apply $\bpartial_{\underline{w}}$ to an element in $\mathfrak{h}^{\ast}$ we obtain a row vector $[z_1,\dots,z_p]\in \F ^{p}=\F\times \cdots \times\F.$ On the other hand when we apply the operator $\bpartial_{\ulw}$ to a column vector $Q$ with entries in $\mathfrak{h}^{\ast},$ we obtain a matrix with entries in $\F.$
%In this section we only use the operator $\bpartial_{\ulw},$  but $\bCpartial_{\ulw}$ will be very useful in the last section of this article.

In the next definition, for any expression $\ulw\in W^{\ast},$ we define an object $\calA_{\ulw}$ in the category $\bcalNT:$

\begin{definition}\label{def-nilalg-associted-to-expression}

If $\ulw=s_{a_1}\cdots s_{a_p}\in W^{\ast},$ we define
%let for each $1 \leq k \leq p,$ $\underline{w}_k=s_{a_1}\cdots s_{a_k},$ then we define the matrix $\mathbf{T}_{\underline{w}}$ as the strict lower triangular matrix of size $p\times p,$ whose  $k^{th}$ row ($1 < k \leq p,$) is given by: $[\bpartial_{\underline{w}_{k-1}}(\alpha_{a_k}),[0]]$ where $[0]$ means to complete the row with zeros at the right:
\begin{enumerate}
\item The strict lower triangular matrix associated to $\ulu$ by:
\begin{equation*}
\mathbf{T}_{\ulw}=\left[\begin{array}{ccccc}
                                   0 & 0 & 0 & \cdots& 0 \\
                                   \partial_{a_1}(\alpha_{a_2}) & 0 & 0 & \cdots& 0 \\
                                   \partial_{a_1}(s_{a_2}\circ\alpha_{a_3}) & \partial_{a_2}(\alpha_{a_3}) & 0 & \cdots& 0 \\
                                   \vdots & \vdots & \ddots & \ddots& \vdots\\
                                   \partial_{a_1}(s_{a_2}\cdots s_{a_{p-1}}\circ \alpha_{a_p}) & \partial_{a_2}(s_{a_3}\cdots s_{a_{p-1}}\circ\alpha_{a_p}) & \cdots& \partial_{a_{p-1}}(\alpha_{a_p}) & 0 \\
                                 \end{array}\right]
\end{equation*}

\item We define the algebra $\calA_{\ulw},$ as the nil graded algebra strictly associated to the matrix $\mathbf{T}_{\ulw}.$ That is $\calA_{\ulw}=\calA(\mathbf{T}_{\ulw}).$
In particular $\calA_{\ulw},$ is the commutative algebra defined by generators $X_1,\dots,X_p,$  and relations:

  \begin{equation*}
    X_1^2=0,\quad X_r^2=\sum_{j<r}\partial_{a_j}(s_{a_{j+1}}\cdots s_{a_{r-1}}\circ \alpha_{a_r})X_jX_r,\quad r=2,\dots p.
  \end{equation*}
\end{enumerate}
\end{definition}

The following lemma shows that we can translate the product $\ulu\ulw$ of two expressions $\ulu,\ulw$ in the monoid $W_{m}^{\ast},$ as a certain $\nabla-$product of the corresponding algebras  $\calA(\ulu)$ and $\calA(\ulw).$

\begin{lemma}\label{lemma-special-nabla-product-for-calA-uw}
  Given two expressions $\ulu,\ulw\in W^{\ast},$ with $\ulu=s_{a_1}\cdots s_{a_p}$ and $\ulw=s_{b_1}\cdots s_{b_l},$
  then we have
  \begin{equation*}
    \calA_{\ulu\ulw}=\calA_{\ulu}\nabla_C\calA_{\ulw}
  \end{equation*}
    %\quad\textrm{and}\quad Q_{\ulw}=
    where
  \begin{equation*}
   C=\bpartial_{\ulu}(Q_{\ulw})\quad\textrm{and}\quad Q_{\ulw}= \left[\begin{matrix}
                   \alpha_{b_1} \\
                   s_{b_1}\circ \alpha_{b_2} \\
                   s_{b_1}s_{b_2}\circ \alpha_{b_3} \\
                   \vdots\\
                   s_{b_1}\cdots s_{b_{l-1}}\circ \alpha_{b_l}
                 \end{matrix}\right].
  \end{equation*}
\end{lemma}
The column vector $Q_{\ulw}$ will be called the \emph{column vector associated} to the expression $\ulw.$

\begin{proof}
  By definition $\calA_{\ulu\ulw}$ is the (abstract) nil graded algebra associated to the triangular matrix $\mathbf{T}_{\ulu\ulw}.$ One can check that:

  \begin{equation*}
\mathbf{T}_{\ulu\ulw}=\left[\begin{matrix}
                          \mathbf{T}_{\ulu} & 0 \\
                          \bpartial_{\ulu}(Q_w) & \mathbf{T}_{\ulw}
                        \end{matrix}\right]
                        =\mathbf{T}_{\ulu}\nabla_{\bpartial_{\ulu}(Q_w)}\mathbf{T}_{\ulw}
\end{equation*}
In fact by definition the $k-$th row of the matrix $\mathbf{T}_{\ulu\ulw}$ is given by
\begin{equation*}
  \left\{\begin{matrix}
           [0,\dots, 0] & \quad\textrm{if}\quad k=1\\
           \quad&\quad\\
           [\bpartial_{\ulu_{k-1}}(\alpha_{a_k}),0,\dots,0] & \quad\textrm{if}\quad 1< k\leq p \\
           \quad&\quad\\
           [\bpartial_{\ulu}(\alpha_{b_1}),0,\dots,0] & \quad\textrm{if}\quad k=p+1 \\
           \quad&\quad\\
           [\bpartial_{\ulu}(\ulw_{k-1}\circ\alpha_{b_{k-p}}),
           \bpartial_{\ulw_{k-1}}(\alpha_{b_{k-p}}),\dots,0] & \quad\textrm{if}\quad k>p+1
         \end{matrix}\right.
\end{equation*}
Therefore the algebra $\calA_{\ulu\ulw}$ is (isomorphic to) the nil graded algebra associated to the matrix $\mathbf{T}_{\ulu}\nabla_{C}\mathbf{T}_{\ulw},$ that is
\begin{equation*}
  \calA_{\ulu\ulw}=\calA(\mathbf{T}_{\ulu}\nabla_{C}\mathbf{T}_{\ulw})
=\calA_{\ulu}\nabla_{C}\calA_{\ulw}.
\end{equation*}
The last equation follows from definition \ref{def-nabla-product-for-algebras}.
\end{proof}

\subsection{Gelfand-Tsetlin subalgebras}\label{ssec-the-Jailbreak-alg}

In this section we study the Gelfand-Tsetlin subalgebras of the many endomorphism algebras coming from category $\calD.$
%Following the nomenclature defined by Libedinsky (see \cite{LibGentle}), \cite{LibLightLeaves}), where cages and trapped birds are very relevant concepts, we extend the terminology and define the \emph{Jailbreak algebras}, assuming each Soergel graph as a kind of \emph{Jail} and their decorative patches are the prisoners trying to escape. The leftmost region will understood as the gate to their freedom.

Given any expression $\ulu\in W^{\ast},$  we denote by $\mathbb{Id}_{\ulu}$ the identity element of $\textrm{End}_{\calD}(\ulu),$ that is, the Soergel graph with top and bottom boundary equal to $\ulu,$ and where vertical lines connect the corresponding boundary points. We also denote by $\mathbb{T}({\ulu})$ the Soergel graph with top and bottom boundary equal to $\ulu,$ and where from each boundary point emerges an edge (with the corresponding color) ending in a univalent vertex. Note that $\mathbb{T}({\ulu})$ only has one region. More particularly, given any expression $\ulu=s_{a_1}\cdots s_{a_p}\in W^{\ast},$ and any $1\leq k \leq p,$ we define the element $\mathbb{J}_k(\ulu)\in\textrm{End}_{\calD}(\ulu)$ as follows:
\begin{equation}\label{eq-def-JM-elements}
  \mathbb{J}_k(\ulu)=\mathbb{Id}_{s_{a_1}\cdots s_{a_{k-1}}}\otimes\mathbb{T}(s_k)\otimes \mathbb{Id}_{s_{a_{k+1}}\cdots s_{a_{p}}} .
\end{equation}

For example if $\underline{u}=({\color{blue}a},{\color{red}b},{\color{blue}a},{\color{green}c},{\color{red}b},{\color{blue}a}),$   then the elements $\mathbb{T}(\ulu)$ and $\mathbb{J}_5(\ulu)$ respectively, look like this:

\begin{equation*}
  \includegraphics[scale=0.2]{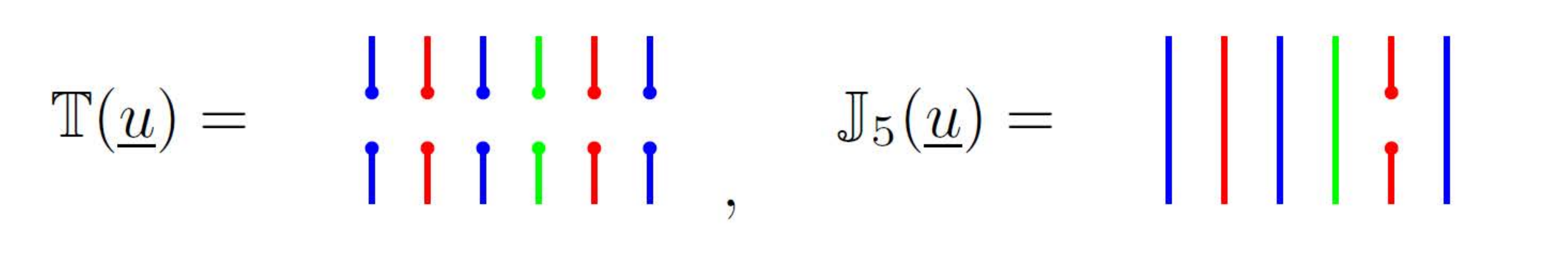}
\end{equation*}

\begin{lemma}\label{lemma-tunel-is-not-zero}
  If $\ulu$ is any expression in $W^{\ast}$, then the element $\mathbb{T}(\ulu)$ is different from zero as morphism of $\calD.$
\end{lemma}
\begin{proof}
  Is not difficult to see that the element $\mathbb{T}(\ulu)$ of any expression $\ulu$ is always a member of the double leaves basis of $\textrm{End}_{\calD}(\ulu)$ (see \cite{LibLightLeaves}).
\end{proof}

\begin{definition}\label{def-jailbreak-algebras}
  Given any expression $\ulu\in W^{\ast}$ we define the algebra $\bcalJ(\ulu)$ as the subalgebra of $\textrm{End}_{\calD}(\ulu),$ generated by all the elements $\mathbb{J}_k(\ulu).$
\end{definition}

\begin{remark}
  Due to the work of Ryom-Hansen \cite{Steen2} we know that the generators of the algebras $\bcalJ(\ulu)$ are Jucys-Murphy elements relative to the Libedinsky's light leaves basis of the many endomorphism algebras coming from category $\calD$ . Therefore, the algebras $\bcalJ(\ulu)$ are the \emph{Gelfand-Tsetlin} subalgebras of those endomorphism algebras.
\end{remark}

 Is not difficult to see that
\begin{equation}\label{eq-Jucys-commute}
  \mathbb{J}_{k}(\ulu)\mathbb{J}_{i}(\ulu)=\mathbb{J}_{i}(\ulu)\mathbb{J}_{k}(\ulu),
\end{equation}

For example if $m=4,$ and $\underline{u}=({\color{green}c},{\color{red}b},{\color{blue}a},{\color{green}c},{\color{blue}a}),$   then

 \begin{equation*}
  \begin{tikzpicture}[xscale=0.3,yscale=0.3]
  \node[] at (-2,1.5) {$\mathbb{J}_2(\ulu)\mathbb{J}_4(\ulu)=$};
  %\draw[thick,blue] (-1,0)--(-1,3);
%\draw[blue,fill] (-1,1) circle [radius=0.1];

%\draw[thick,red] (0,0)--(0,3);
%\draw[red,fill] (0,1) circle [radius=0.1];

%\draw[thick,blue] (1,0)--(1,3);
%\draw[blue,fill] (1,1) circle [radius=0.1];

\draw[thick,green] (2,0)--(2,3);
%\draw[green,fill] (2,1) circle [radius=0.1];

\draw[thick,red] (3,0)--(3,0.5);
\draw[red,fill] (3,0.5) circle [radius=0.1];

\draw[thick,blue] (4,0)--(4,3);
%\draw[blue,fill] (4,1) circle [radius=0.1];

\draw[thick,green] (5,0)--(5,2);
\draw[green,fill] (5,2) circle [radius=0.1];

\draw[thick,blue] (6,0)--(6,3);

%%%%%%%%%%%%%%%%%%%%%%%%%%%%

\draw[thick,red] (3,1)--(3,3);
\draw[red,fill] (3,1) circle [radius=0.1];

\draw[thick,green] (5,3)--(5,2.5);
\draw[green,fill] (5,2.5) circle [radius=0.1];

%\draw[thin,gray] (-2,3)--(11.5,3);
%\draw[thin,gray] (-2,0)--(11.5,0);

%%%%%%%%%%%%%%%%%%%%%%%%%%%%%%%%
%%%%%%%%%%%%%%%%%%%%%%%%%%%%%%%%%%

\node[] at (7.5,1.5) {$=$};

\draw[thick,green] (2+7,0)--(2+7,3);
%\draw[green,fill] (2,1) circle [radius=0.1];

\draw[thick,red] (3+7,0)--(3+7,2);
\draw[red,fill] (3+7,2) circle [radius=0.1];

\draw[thick,blue] (4+7,0)--(4+7,3);
%\draw[blue,fill] (4,1) circle [radius=0.1];

\draw[thick,green] (5+7,0)--(5+7,0.5);
\draw[green,fill] (5+7,0.5) circle [radius=0.1];

\draw[thick,blue] (6+7,0)--(6+7,3);

%%%%%%%%%%%%%%%%%%%%%%%%%%%%

\draw[thick,red] (3+7,2.5)--(3+7,3);
\draw[red,fill] (3+7,2.5) circle [radius=0.1];

\draw[thick,green] (5+7,3)--(5+7,1);
\draw[green,fill] (5+7,1) circle [radius=0.1];

%\draw[thin,gray] (-2+15+0.5,3)--(12+15,3);
%\draw[thin,gray] (-2+15+0.5,0)--(12+15,0);

\node[] at (13+5,1.5) {$=\mathbb{J}_4(\ulu)\mathbb{J}_2(\ulu).$};

\end{tikzpicture}
\end{equation*}

Therefore, $\bcalJ(\ulu)$ is a commutative algebra.

Also note that

\begin{equation}\label{eq-Jucys-generate-Tunnel}
  \prod_{k=1}^{l^{\ast}(\ulu)}\mathbb{J}_{k}(\ulu)=\mathbb{T}(\ulu)\neq 0,
\end{equation}
%then we always can find a tunnel $\mathbb{T}(\ulu)$ for each jailbreak algebra $\bcalJ(\ulu)$

%Particularly the product of the generators $\prod\mathbb{J}_{k}(\ulu)$ is different from zero in $\bcalJ(\ulu).$

Finally, note that by definition, $\deg(\mathbb{J}_k(\ulu))=2,$ then the algebra $\bcalJ(\ulu)$ is positively graded by even numbers.

\begin{lemma}\label{lemma-action-of-word-ID-graph-on-f}
  Given any expression $\ulu=s_{a_1}\cdots s_{a_p}\in W^{\ast},$ and any homogeneous element $f\in\bcalR$ with $\deg(f)\geq2,$ then the following equation holds in $\calD:$
  \begin{equation}\label{eq-lemma-action-of-word-ID-graph-on-f-modificada}
    \mathbb{Id}_{\ulu}f=\sum_{k=1}^{p}C_k(f)\mathbb{J}_{k}(\ulu),
  \end{equation}
  where $C_k(f)=\partial_{a_k}(s_{a_{k+1}}\cdots s_p \circ f).$
\end{lemma}
\begin{proof}
  We use induction on $p=l^{\ast}(\ulu)$ to prove first that
   \begin{equation}\label{eq-lemma-action-of-word-ID-graph-on-f}
    \mathbb{Id}_{\ulu}f=(\ulu\circ f)\mathbb{Id}_{\ulu}+\sum_{k=1}^{p}C_k(f)\mathbb{J}_{k}(\ulu),
  \end{equation}

  If $p=1$ then equation \ref{eq-lemma-action-of-word-ID-graph-on-f} looks like:
  \begin{equation*}
    \mathbb{Id}_{s_{a_1}}f=(s_{a_1}\circ f) \mathbb{Id}_{s_{a_1}}+\partial_{a_1}(f)\mathbb{J}_{1}(s_{a_1}),
  \end{equation*}
  and this follows directly from relation \ref{eq-polynomial-relation-2}. If $p>1$ then we can apply relation \ref{eq-polynomial-relation-2} to obtain
  \begin{equation*}
    \mathbb{Id}_{\ulu}f=\calS+\partial_{a_p}(f)\mathbb{J}_{p}(\ulu),
  \end{equation*}
  where $\calS=(\mathbb{Id}_{\ulu_{p-1}}s_{a_p}\circ f)\otimes \mathbb{Id}_{s_{a_p}}$

   \begin{equation*}
  \begin{tikzpicture}[xscale=0.5,yscale=0.4]
  \node[] at (-5,1.5) {$\calS=$};
  \draw[thick] (-4,0)--(-4,3);
  \node[below] at (-4,0) {$a_1$};
  \node[] at (-2.5,1.5) {$\cdots$};
  \draw[thick] (-1,0)--(-1,3);
  \node[below] at (-1,0) {$a_{p-1}$};
 %\draw[thick] (0,0)--(0,3);

 \node[] at (1,1.5) {$s_{a_p}\circ f$};
 \draw[thin] (-0.2,2)--(2.2,2);
 \draw[thin] (-0.2,1)--(2.2,1);
 \draw[thin] (-0.2,2)--(-0.2,1);
 \draw[thin] (2.2,1)--(2.2,2);
%\draw[thick,blue] (1,0)--(1,1);
%\draw[fill,blue] (1,1) circle [radius=0.1];
%\draw[thick,blue] (1,3)--(1,2);
%\draw[fill,blue] (1,2) circle [radius=0.1];

\draw[thick] (3,0)--(3,3);
\node[below] at (3,0) {$a_p$};

\end{tikzpicture}
\end{equation*}
Now by induction
\begin{equation*}
    \mathbb{Id}_{\ulu_{p-1}}(s_{a_p}\circ f)=\ulu_{p-1}\circ (s_{a_p}\circ f)\mathbb{Id}_{\ulu_{p-1}}+\sum_{k=1}^{p-1}C_k(s_{a_p}\circ f)\mathbb{J}_{k}(\ulu_{p-1}),
  \end{equation*}
 where
 $$C_k(s_{a_p}\circ f)=\partial_{a_k}(s_{a_{k+1}}\cdots s_{p-1}\circ(s_{a_p}\circ f))=\partial_{a_k}(s_{a_{k+1}}\cdots s_{p-1}s_{a_p}\circ f).$$

Finally relation \ref{eq-left-annihilation} implies that
$   \mathbb{Id}_{\ulu}f=\sum_{k=1}^{p}C_k(f)\mathbb{J}_{k}(\ulu)$ as desired.
\end{proof}

\begin{theorem}\label{theo-Ju-isomorphic-to-ATu}
  The algebra $\bcalJ(\ulu)$ is isomorphic to the nil-graded algebra $\calA_{\ulu}.$ In particular the algebra $\bcalJ({\ulu}),$ is the abstract commutative algebra given by generators $\mathbb{J}_1,\dots,\mathbb{J}_l,$ ($l=l^{\ast}(\ulu)$ ) and relations:

  \begin{equation}\label{eq-theo-Ju-isomorphic-to-ATu}
    \mathbb{J}_1^2=0,\quad \mathbb{J}_r^2=\sum_{j<r}\partial_{a_j}(s_{a_{j+1}}\cdots s_{a_{r-1}}\circ \alpha_{a_r})\mathbb{J}_j\mathbb{J}_r,\quad r=2,\dots l.
  \end{equation}
\end{theorem}

\begin{proof}
Let $\ulu=s_{a_1}\cdots s_{a_p}\in W^{\ast}.$

  First note that by the diagrammatic relations of the category $\calD,$ we have that
  \begin{equation*}
    (\mathbb{J}_{k}(\ulu))^{2}=(\mathbb{Id}_{\ulu_{k-1}}\alpha_{a_k})\otimes \mathbb{T}(s_{a_k})\otimes\mathbb{Id}_{\ulu_{>k}}
  \end{equation*}
  where $\ulu_{>k}=s_{a_{k+1}}\cdots s_{a_p}:$

  \begin{equation*}
  \begin{tikzpicture}[xscale=0.4,yscale=0.25]
  \node[] at (-7,0) {$(\mathbb{J}_{k}(\ulu))^{2}=$};
  \draw[thick] (-4,-3)--(-4,3);
  \node[below] at (-4,-3) {$a_1$};
  \node[] at (-2.5,0) {$\cdots$};
  \draw[thick] (-1,-3)--(-1,3);
  \node[below] at (-1,-3) {${ a_{k-1}}$};
 %\draw[thick] (0,0)--(0,3);

 %\node[] at (1,1.5) {$\alpha_{a_p}\cdot f$};
 %\draw[thin] (-0.2,2)--(2.2,2);
 %\draw[thin] (-0.2,1)--(2.2,1);
 %\draw[thin] (-0.2,2)--(-0.2,1);
 %\draw[thin] (2.2,1)--(2.2,2);
 \node[below] at (1,-3) {${\color{blue}a_{k}}$};
\draw[thick,blue] (1,0)--(1,1);
\draw[fill,blue] (1,1) circle [radius=0.1];
\draw[thick,blue] (1,3)--(1,2);
\draw[fill,blue] (1,2) circle [radius=0.1];
\draw[thick,blue] (1,0)--(1,-1);
\draw[fill,blue] (1,-1) circle [radius=0.1];
\draw[thick,blue] (1,-3)--(1,-2);
\draw[fill,blue] (1,-2) circle [radius=0.1];

\draw[thick] (3,-3)--(3,3);
\node[below] at (3,-3) {$a_{k+1}$};
\node[] at (4,0) {$\cdots$};
\node[below] at (5.5,-3) {$a_{p}$};
\draw[thick] (5.5,-3)--(5.5,3);

\node[] at (6.7,0) {$=$};

%%%%%%%%%%%%%%%%%%%%%%%%%%%

\draw[thick] (-4+12,-3)--(-4+12,3);
  \node[below] at (-4+12,-3) {$a_1$};
  \node[] at (-2.5+12,0) {$\cdots$};
  \draw[thick] (-1+12,-3)--(-1+12,3);
  \node[below] at (-1+12,-3) {$a_{k-1}$};
 %\draw[thick] (0,0)--(0,3);

 \node[] at (1+12,0) {$\alpha_{{\color{blue}a_k}}$};
 \draw[thin] (-0.2+12,-1)--(2.2+12,-1);
 \draw[thin] (-0.2+12,1)--(2.2+12,1);
 \draw[thin] (-0.2+12,-1)--(-0.2+12,1);
 \draw[thin] (2.2+12,1)--(2.2+12,-1);
 \node[below] at (1+12,-3) {${\color{blue}a_{k}}$};
%\draw[thick,blue] (1+12,0)--(1+12,1);
%\draw[fill,blue] (1+12,1) circle [radius=0.1];
\draw[thick,blue] (1+12,3)--(1+12,2);
\draw[fill,blue] (1+12,2) circle [radius=0.1];
%\draw[thick,blue] (1+12,0)--(1+12,-1);
%\draw[fill,blue] (1+12,-1) circle [radius=0.1];
\draw[thick,blue] (1+12,-3)--(1+12,-2);
\draw[fill,blue] (1+12,-2) circle [radius=0.1];

\draw[thick] (3+12,-3)--(3+12,3);
\node[below] at (3+12,-3) {$a_{k+1}$};
\node[] at (4+12,0) {$\cdots$};
\node[below] at (5.5+12,-3) {$a_{p}$};
\draw[thick] (5.5+12,-3)--(5.5+12,3);
\end{tikzpicture}
\end{equation*}
  In particular, if $k=1$ we have $(\mathbb{J}_{k}(\ulu))^{2}=0.$

If $k>1$ we have:
\begin{equation*}
    (\mathbb{J}_{k}(\ulu))^{2}=(\sum_{i<k}C_i\mathbb{J}_{i}(\ulu_{k-1}))\otimes \mathbb{T}(s_{a_k})\otimes\mathbb{Id}_{\ulu_{>k}}
    =\sum_{i<k}C_i\mathbb{J}_{i}(\ulu)\mathbb{J}_{k}(\ulu).
  \end{equation*}
  where $C_i=\partial_{a_i}(s_{a_{i+1}}\cdots s_{a_{k-1}}\circ \alpha_{a_k})$ and they are exactly the first $k-1$ coefficients in the $k^{th}$ row of the matrix $\mathbf{T}_{\ulu}$ (the other are zeros). The last calculus, together with equation \ref{eq-Jucys-commute}, imply that the algebra $\bcalJ(\ulu)$ is a commutative nil graded algebra  weakly associated to the matrix $\mathbf{T}_{\ulu}.$ The conclusion of the theorem comes from equation \ref{eq-Jucys-generate-Tunnel} and lemma \ref{lemma-tunel-is-not-zero} and lemma \ref{lemma-monomial-basis}.
\end{proof}

%\begin{remark}
  Theorem \ref{theo-Ju-isomorphic-to-ATu} asserts that the relations \ref{eq-theo-Ju-isomorphic-to-ATu} coming from the local relations in the definition of the category $\calD,$ not only are true, they are enough to provide an explicit presentation of the algebras $\bcalJ(\ulu).$  In particular, as a direct consequence of theorem \ref{theo-Ju-isomorphic-to-ATu}  we have that the set $\{\mathbb{J}_{1}^{a_1}\cdots\mathbb{J}_{l}^{a_l}:a_j\in\{0,1\}\}$ is a basis of the algebra $\bcalJ(\ulu),$ and then $\dim(\bcalJ(\ulu))=2^{l},$ where $l=l^{\ast}(\ulu).$
%\end{remark}

\begin{corollary}\label{coro-theo-Ju-isomorphic-to-ATu}
  Given two expressions $\ulu,\ulw\in W^{\ast}$ we have that
  \begin{equation*}
    \bcalJ(\ulu\ulw)=\bcalJ(\ulu)\nabla_C\bcalJ(\ulw)
  \end{equation*}
  where $C=\bpartial_{\ulu}(Q_{\ulw}),$ and $Q_{\ulw}$ is the column vector associated to the expression $\ulw.$
\end{corollary}
\begin{proof}
  Is a direct consequence of theorem \ref{theo-Ju-isomorphic-to-ATu} and lemma \ref{lemma-special-nabla-product-for-calA-uw}.
\end{proof}

\subsection{Explicit calculus in type $\widetilde{A}$}\label{ssec-explicit-for-type-A}
In this section we develop an algorithm, based on Theorem \ref{theo-Ju-isomorphic-to-ATu},  to obtain explicit presentations for the many Gelfand-Tsetlin subalgebras $\bcalJ(\ulu),$ in the case of type $\widetilde{A}.$ By theorems \ref{theo-Ju-isomorphic-to-ATu} to obtain a presentation for the algebra $\bcalJ(\ulu),$ is equivalent to calculate the associated matrix $\mathbf{T}_{\ulu}.$
On the other hand by corollary \ref{coro-theo-Ju-isomorphic-to-ATu}, we can concentrate in to calculate explicitly the matrix $\mathbf{T}_{\ulw}$ for certain \emph{generating} expressions $\ulw$ and then to describe how we can assemble them via an appropriate $\nabla-$product, to obtain any other  matrix $\mathbf{T}_{\ulu}.$

Let $(W_m,S_m)$ the Coxeter system of type $\widetilde{A}_{m-1}$. We consider the \emph{geometric realization} of $(W_m,S_m),$  defined by:a

\begin{equation}\label{definition-action-root-coroot}
  \langle \alpha_a,\alpha^{\vee}_b \rangle=-2\cos(\pi/{\mathbf{m}(a,b)})
 % =\left\{\begin{array}{cc}
 %                                            2 & \quad\textrm{if}\quad a=b \\
 %                                            -1 & \quad\textrm{if}\quad a=b\pm 1 \\
 %                                            0 & \quad\textrm{otherwise}
 %                                          \end{array}\right.
\end{equation}

%The \emph{Cartan Matrix} of the realization is defined as $\mathcal{C}=[\langle \alpha_a,\alpha^{\vee}_b \rangle]$

%For three given elements $a,b,c\in I_m$ such that $b\neq a,c$ we denote:

% \begin{equation}\label{eq-up-down-interval}
%   \lfloor a,b,c\rceil=\lfloor a,b\rfloor\lceil b-1,c\rceil,\quad \textrm{and}\quad \lceil a,b,c\rfloor=\lceil a,b\rceil\lfloor b+1,c\rceil.
% \end{equation}

%One can easily check, that the expressions in equations \ref{eq-up-interval-coxeter} are reduced expressions for the corresponding elements in  $ W_m.$ It is well known that the Coxeter group of type $\widetilde{A}_{m-1},$ is infinite. In particular, note that the element $\lfloor a,a-1 \rfloor\in W_m,$ has infinite order.
%One can easily check that the group $W_m$ is generated by the set of all intervals.

In this case, the action of $W_{m}^{\ast}$ and $W_m$ on $\bcalR=S(\mathfrak{h}^{\ast})$ is given by:
\begin{equation}\label{definition-dual-geometric-representation-TA}
  s_{a}\circ \alpha_b=\alpha_b\cdot s_a=\left\{\begin{array}{cc}
                                             -\alpha_b & \quad\textrm{if}\quad a=b \\
                                             \alpha_a+\alpha_b & \quad\textrm{if}\quad a=b\pm 1 \\
                                             \alpha_b & \quad\textrm{otherwise}
                                           \end{array}\right.
\end{equation}

And the Demazure operator by

\begin{equation}\label{Demazure-type-A-tilde-m}
  \partial_{a}(\alpha_b)=\left\{\begin{array}{cc}
                                             2 & \quad\textrm{if}\quad a=b \\
                                             -1 & \quad\textrm{if}\quad a=b\pm 1 \\
                                             0 & \quad\textrm{otherwise}
                                           \end{array}\right.
\end{equation}

Note that:

\begin{equation}\label{eq-sum-of-roots-is-zero}
  \sum_{a\in I_m}{\alpha_a}=0.
\end{equation}

Given $a, b\in I_m,$ we define the intervals $\lfloor a,b\rfloor,\lceil a,b\rceil\in W_m^{\ast}$, as follows:

 \begin{equation}\label{eq-up-interval-coxeter}
   \lfloor a,b\rfloor=\left\{\begin{array}{cc}
                               s_{a}s_{a+1}\cdots s_{b} & \quad \textrm{if}\quad b\neq a \\
                               s_a & \quad \textrm{if}\quad b= a
                             \end{array} \right. \quad \textrm{and}\quad
   \lceil a,b\rceil=\left\{\begin{array}{cc}
                                s_{a}s_{a-1}\cdots s_{b} & \quad \textrm{if}\quad b\neq a \\
                               s_a & \quad \textrm{if}\quad b= a
                             \end{array} \right.
 \end{equation}

We also define the elements $\balpha(\lfloor a, b \rfloor),\balpha(\lceil a, b \rceil)\in \bcalR$ as follows:

\begin{equation}\label{def-notation-alpha-intervals}
\begin{array}{c}
 \balpha(\lfloor a, b \rfloor)=\alpha_{a}+\alpha_{a+1}+\cdots +\alpha_{b} \quad \textrm{and}\quad
  \balpha(\lceil a, b \rceil)=\alpha_{a}+\alpha_{a-1}+\cdots +\alpha_{b}
 \end{array}
\end{equation}

In particular, by equation \ref{eq-sum-of-roots-is-zero}, we have

\begin{equation}\label{eq-sum-of-roots-is-zero2}
  \balpha(\lfloor a, a-1 \rfloor)=\balpha(\lceil a, a+1 \rceil)=0.
\end{equation}

%Our plan is to learn how we can obtain the presentations of the algebras  $\bcalJ(\ulw)$ for any interval $\ulw,$ and then develop an algorithm to

In order to optimize our calculus, we will introduce a new concept, the \emph{extended matrix} associated to $\ulu$ and \emph{the special  $\bnabla-$product} for extended matrices, as follows:
%In this case we will add an extra column (the column $0$ at the left of the matrix $\mathbf{T}_{\ulu}$). This column will have as entries some elements of $\mathfrak{h}^{\ast},$ associated to the expression $\ulu.$ The extended matrix $\widetilde{\mathbf{T}}_{\ulu}$ looks like this $[Q_{\ulu},\mathbf{T}_{\ulu}]$ where $Q_{\ulu}$ is a convenient column vector.
%use lemmas \ref{lemma-commuting-move-for-type-Atilde} and \ref{lemma-T-matrix-for-intervals} and a particular case of the $\nabla-$products. The main idea is to describe the algebra $\bcalJ(\ulu)$ as $\nabla-$product of algebras $\bcalJ(\ulw),$ where $\ulw$ are appropriate intervals.

%Now we need to define a

\begin{definition}\label{def-extended-Tmatrix-associated-to-expression}
Given two expressions $\ulu,\ulw\in W^{\ast}_m,$
\begin{enumerate}
\item We define the \emph{extended matrix} $\widetilde{\mathbf{T}}_{\ulu}$ by:
\begin{equation*}
  \widetilde{\mathbf{T}}_{\ulu}=[Q_{\ulu},\mathbf{T}_{\ulu}].
\end{equation*}
%let for each $1 \leq k \leq l,$ $\underline{w}_k=s_{a_1}\cdots s_{a_k},$
%then we define $\mathbf{T}_{\ulu}$ is the matrix of definition \ref{def-Tmatrix-associated-to-expression} and
where $Q_{\ulu}$ is the column vector associated to $\ulu$ (see lemma \ref{lemma-special-nabla-product-for-calA-uw}) and $\mathbf{T}_{\ulu}$ is the strict lower triangular matrix associated to $\ulu.$
\item We define:
  \begin{equation*}
    \widetilde{\mathbf{T}}_{\ulu}\bnabla \widetilde{\mathbf{T}}_{\ulw}
  =[Q^{\ulu}_{\ulw},\mathbf{T}_{\ulu}\bnabla \mathbf{T}_{\ulw}].
  \end{equation*}
where
\begin{equation*}
 \mathbf{T}_{\ulu}\bnabla \mathbf{T}_{\ulw}=\mathbf{T}_{\ulu}\nabla_{\bCpartial_{\ulu}(Q_{\ulw})} \mathbf{T}_{\ulw}
 \quad\textrm{and}\quad Q^{\ulu}_{\ulw}=\left[\begin{matrix}
             Q_{\ulu}  \\
            \ulu \circ Q_{\ulw}
          \end{matrix}\right]
\end{equation*}
Therefore
\begin{equation*}
 \widetilde{\mathbf{T}}_{\ulu}\bnabla \widetilde{\mathbf{T}}_{\ulw}= [Q_{\ulu},{\mathbf{T}}_{\ulu}]\bnabla[Q_{\ulw},{\mathbf{T}}_{\ulw}]=
    \left[\begin{matrix}
             Q_{\ulu} & \mathbf{T}_{\ulu} & 0 \\
            \ulu \circ Q_{\ulw} & \bpartial_{\ulu}(Q_{\ulw}) & {\mathbf{T}}_{\ulw}\end{matrix}\right].
\end{equation*}
\end{enumerate}
\end{definition}

\begin{lemma}\label{lemma-associativity-bnabla-product}
  Given three expressions $\ulu,\ulv,\ulw\in W^{\ast}_m,$ we have that
  \begin{equation*}
    (\widetilde{\mathbf{T}}_{\ulu}\bnabla \widetilde{\mathbf{T}}_{\ulv})\bnabla\widetilde{\mathbf{T}}_{\ulw}=
    \widetilde{\mathbf{T}}_{\ulu}\bnabla (\widetilde{\mathbf{T}}_{\ulv}\bnabla\widetilde{\mathbf{T}}_{\ulw})
  \end{equation*}
\end{lemma}

\begin{proof}
  We have to prove that
  \begin{equation*}
    Q^{\ulu\ulv}_{\ulw}=Q^{\ulu}_{\ulv\ulw}\quad\textrm{and}\quad
    (\mathbf{T}_{\ulu}\bnabla\mathbf{T}_{\ulv})\bnabla \mathbf{T}_{\ulw}=
    \mathbf{T}_{\ulu}\bnabla(\mathbf{T}_{\ulv}\bnabla \mathbf{T}_{\ulw}).
  \end{equation*}
  The equation $Q^{\ulu\ulv}_{\ulw}=Q^{\ulu}_{\ulv\ulw}$ follows from the definition of the column vector associated to an expression (see lemma \ref{lemma-special-nabla-product-for-calA-uw} and the associativity of the monoid $W^{\ast}_{m}.$

  The equation  $(\mathbf{T}_{\ulu}\bnabla\mathbf{T}_{\ulv})\bnabla \mathbf{T}_{\ulw}=
    \mathbf{T}_{\ulu}\bnabla(\mathbf{T}_{\ulv}\bnabla \mathbf{T}_{\ulw})$ is a particular case of lemma \ref{lemma-associativity-nabla-prod-matrix}.
\end{proof}

By lemmas \ref{lemma-special-nabla-product-for-calA-uw}, \ref{lemma-associativity-bnabla-product} and corollary \ref{coro-theo-Ju-isomorphic-to-ATu} we have that the algebra $\bcalJ(\ulu\ulw)$ is isomorphic to the nil graded algebra $\calA(\mathbf{T}_{\ulu}\bnabla \mathbf{T}_{\ulw}).$ More generally if we write the expression $\ulu$ as a product of intervals $\ulu=\ulv_1\cdots \ulv_k,$ we will show that the algebra $\bcalJ(\ulu)$ is isomorphic to the nil graded algebra $\calA(\mathbf{T}_{\ulv_1}\bnabla\cdots \bnabla \mathbf{T}_{\ulv_k}).$ Therefore to obtain a presentation for  $\bcalJ(\ulu),$ we only need to develop the $\bnabla-$product of matrices $\mathbf{T}_{\ulv_1}\bnabla\cdots \bnabla \mathbf{T}_{\ulv_k}.$ Note that in each step we need to calculate the corresponding column vectors $Q^{\ulv_j}_{\ulv_{j+1}}.$ Therefore the special $\bnabla-$product of extended matrices arises as a systematical way to manage all this information. Our idea is to work with decomposition of expressions as product of intervals, then it would be useful to have an explicit formula for $\widetilde{\mathbf{T}}_{\ulw}$ whenever $\ulw$ is an interval. The next lemma describe the different cases:

%More generally we have for intervals the following result:

\begin{lemma}\label{lemma-matrix-T-for-interval-any-type}
Let $a,b\in I_m,$
\begin{enumerate}
  \item If $a=b$ then $\widetilde{\mathbf{T}}_{\lfloor a,b\rfloor}=\widetilde{\mathbf{T}}_{\lceil a,b \rceil}=[\alpha_a,0].$
  \item If $l^{\ast}(\lfloor a,b\rfloor)<m$ (resp. if $l^{\ast}(\lceil a,b \rceil)<m),$ then
  \begin{equation*}
  \widetilde{\mathbf{T}}_{\lfloor a,b\rfloor}=\left[\begin{array}{cccccc}
                             \alpha_{a}& 0 & 0 & 0 & 0  \\
                             \balpha(\lfloor a,a+1\rfloor)&-1 & 0 & 0 & \cdots & 0 \\
                             \balpha(\lfloor a,a+2\rfloor)&-1 & -1 & 0 & \cdots & 0\\
                             \vdots& \vdots & \vdots & \ddots & \ddots&\vdots \\
                             \balpha(\lfloor a,b\rfloor)&-1 & \cdots & -1& -1 & 0
                          \end{array}\right].
\end{equation*}
and respectively
\begin{equation*}
  \widetilde{\mathbf{T}}_{\lceil a,b\rceil}=\left[\begin{array}{cccccc}
                                  \alpha_{a}& 0 & 0 & 0 & 0 & 0\\
                                 \balpha(\lceil a,a-1\rceil)& -1 & 0 & 0 & \cdots & 0 \\
                                 \balpha(\lceil a,a-2\rceil)& -1 & -1 & 0 & \cdots & 0 \\
                                 \vdots& \vdots & \vdots & \ddots & \ddots&\vdots \\
                                  \balpha(\lceil a,b\rceil)&-1 & \cdots & -1 & -1 & 0 \\
                                  %\alpha_{0}& -2 & -1 & -1 &-1 & 0
                                 \end{array}\right].
\end{equation*}

  \item We also have
  \begin{equation*}
  \widetilde{\mathbf{T}}_{\lfloor a,a-1\rfloor}=\left[\begin{array}{cccccc}
                             \alpha_{a}& 0 & 0 &0& \cdots & 0  \\
                             \balpha(\lfloor a,a+1\rfloor)&-1 & 0& 0 &\cdots& 0 \\
                             \balpha(\lfloor a,a+2\rfloor)&-1 & -1 &0& \cdots& 0 \\
                             \vdots& \vdots & \vdots & \ddots & \ddots&\vdots \\
                             \balpha(\lfloor a,a-2\rfloor)&-1 & \cdots & -1  & 0& 0 \\
                             \alpha_{a}& -2 & -1 & \cdots-1 & -1&0
                          \end{array}\right]
\end{equation*}
and
\begin{equation*}
  \widetilde{\mathbf{T}}_{\lceil a,a+1\rceil}=\left[\begin{array}{cccccc}
                                  \alpha_{a}& 0 & 0 & 0 & \cdots & 0\\
                                 \balpha(\lceil a,a-1\rceil)& -1 & 0 & 0 & \cdots & 0 \\
                                 \balpha(\lceil a,a-2\rceil)& -1 & -1 & 0 & \cdots & 0 \\
                                 \vdots& \vdots & \vdots & \ddots & \ddots & \vdots \\
                                  \balpha(\lceil a,a+2\rceil)&-1 & \cdots & -1  & 0& 0 \\
                                  \alpha_{a}& -2 & -1 & \cdots-1 & -1&0
                                 \end{array}\right].
\end{equation*}
respectively.
\end{enumerate}
\end{lemma}
 \begin{proof}
 The case $m=2$ is trivial. Lets we assume that $m>2.$

  The first assertion is trivial.

  The other two assertions follows by induction on length $l$ of the corresponding interval. We will work with ascending intervals only, the other case is analogous.

  If $\ulw=\lfloor a, a+1 \rfloor=s_as_{a+1}$ then by definition of $\mathbf{T}_{\ulw}$ is given by $t_{21}=\partial_a(s_{a+1})=-1$ (the other entries are equal to $0$). The case %$\ulw=\lceil a, b \rceil$ is analogous.

  If $2<l=l^{\ast}(\ulw)<m$ say $\ulw=\lfloor a, b \rfloor=\lfloor a, b-1 \rfloor s_b$ .  By definition of $\mathbf{T}_{\ulw},$ the first  $l-1$ rows of $\mathbf{T}_{\ulw}$ are (basically) given by the matrix  $\mathbf{T}_{\ulv},$ where $\ulv=\lfloor a, b-1 \rfloor.$ By induction the entries $t_{kj}$ of $\mathbf{T}_{\ulv}$ are equal to $-1$ whenever $k>j$ (the other are $0$). The last row of $\mathbf{T}_{\ulw}$ is defined by
  $$[\partial_a(\lfloor a+1,b-1\rfloor\circ\alpha_b),\partial_{a+1}(\lfloor a+2,b-1\rfloor\circ\alpha_b),\dots,\partial_{b-2}(s_{b-1}\circ \alpha_b),\partial_{b-1}(\alpha_b),0].$$

  Using recursively equations \ref{definition-dual-geometric-representation-TA} and \ref{Demazure-type-A-tilde-m}, we can check that:

  \begin{equation*}
    \left\{\begin{matrix}
      \partial_{b-1}(\alpha_b)=-1; & s_{b-1}\circ \alpha_b=\balpha(\lfloor b-1,b\rfloor)\\
      \quad&\quad\\
       \partial_{b-2}(\balpha(\lfloor b-1,b\rfloor))=-1; & \lfloor b-2,b-1\rfloor\circ\alpha_b=\balpha(\lfloor b-2,b\rfloor)\\
       \quad&\quad\\
       \partial_{b-3}(\balpha(\lfloor b-2,b\rfloor))=-1; & \lfloor b-3,b-1\rfloor\circ\alpha_b=\balpha(\lfloor b-3,b\rfloor)\\
      \vdots & \vdots\\
    \end{matrix}\right.
  \end{equation*}
  In general
  $\partial_{c}(\balpha(\lfloor c+1,b\rfloor))=-1$ and $\lfloor c,b-1\rfloor\circ\alpha_b=\balpha(\lfloor c,b\rfloor)$ for any $c=a,\dots,b-2.$
  Therefore we obtain the desired conclusion.

  The case $l=m$ follows analogously. We only have to notice that in this case
  $t_{m1}=\partial_{a}(\balpha(\lfloor a+1,a-1\rfloor))=-2$ and $\lfloor a,a-2\rfloor\circ \alpha_{a-1}=\balpha(\lfloor a,a-1\rfloor)+\alpha_{a}=\alpha_{a}$
\end{proof}

\begin{example}\label{ex-extended-Tmatrix-associated-to-interval-type2}
If $m=5,$  $\ulu=\lfloor 0,3 \rfloor$ and $\ulw=\lceil 0,1\rceil,$ then we have:
\begin{equation*}
  \widetilde{\mathbf{T}}_{\ulu}=\left[\begin{array}{ccccc}
                             \alpha_{0}& 0 & 0 & 0 & 0  \\
                             \balpha(\lfloor0,1\rfloor)&-1 & 0 & 0 & 0  \\
                             \balpha(\lfloor0,2\rfloor)& -1 & -1 & 0 & 0 \\
                             \balpha(\lfloor0,3\rfloor)&-1 & -1 & -1 & 0
                          \end{array}\right],\quad
\widetilde{\mathbf{T}}_{\underline{w}}=\left[\begin{array}{cccccc}
                                  \alpha_{0}& 0 & 0 & 0 & 0 & 0\\
                                 \balpha(\lceil0,4\rceil)& -1 & 0 & 0 & 0 & 0 \\
                                 \balpha(\lceil0,3\rceil)& -1 & -1 & 0 & 0 & 0\\
                                  \balpha(\lceil0,2\rceil)&-1 & -1 & -1 & 0 & 0 \\
                                  \alpha_{0}& -2 & -1 & -1 &-1 & 0
                                 \end{array}\right].
\end{equation*}
\end{example}

As we can see in the proof of lemma \ref{lemma-matrix-T-for-interval-any-type}, most of the time we will need to calculate the action of $s_c$ and the Demazure operator $\partial_c$ over elements of the forms $\balpha(\lfloor a,b\rfloor)$ and $\balpha(\lceil a,b \rceil).$ In order to obtain another level of optimization in the process of calculating the matrix $\mathbf{T}_{\ulu},$ we add the following lemma

\begin{lemma}\label{lemma-tables-action-vs-Demazure}
  Let $a,b,c\in I_m.$
  \begin{enumerate}
  \item If $l^{\ast}(\lfloor a,b\rfloor)<m-1,$ then we have the following table:
  \begin{equation*}
    \begin{array}{ccc}
    s_c\circ \balpha(\lfloor a,b\rfloor), & \partial_c( \balpha(\lfloor a,b\rfloor)),&\textrm{Condition} \\
    \quad & \quad \\
    \balpha(\lfloor a-1,b\rfloor), & -1,& \textrm{if}\quad c=a-1 \\
    \quad & \quad \\
    \balpha(\lfloor a,b+1\rfloor), & -1,&\textrm{if}\quad c=b+1 \\
    \quad & \quad \\
    \balpha(\lfloor a+1,b\rfloor), &1,& \textrm{if}\quad c=a \\
    \quad & \quad \\
    \balpha(\lfloor a,b-1\rfloor), &1,& \textrm{if}\quad c=b \\
    \quad & \quad \\
    \balpha(\lfloor a,b\rfloor), &0,& \textrm{otherwise} \\
    \quad & \quad
    \end{array}
  \end{equation*}
  \item If $l^{\ast}(\lceil a,b\rceil)<m-1,$ then we have the following table:
  \begin{equation*}
    \begin{array}{ccc}
    s_c\circ \balpha(\lceil a,b\rceil), & \partial_c( \balpha(\lceil a,b\rceil)),&\textrm{Condition} \\
    \quad & \quad \\
    \balpha(\lceil a+1,b\rceil), & -1,& \textrm{if}\quad c=a+1 \\
    \quad & \quad \\
    \balpha(\lceil a,b-1\rceil), & -1,&\textrm{if}\quad c=b-1 \\
    \quad & \quad \\
    \balpha(\lceil a-1,b\rceil), &1,& \textrm{if}\quad c=a \\
    \quad & \quad \\
    \balpha(\lceil a,b+1\rceil), &1,& \textrm{if}\quad c=b \\
    \quad & \quad \\
    \balpha(\lceil a,b\rceil), &0,& \textrm{otherwise} \\
    \quad & \quad
    \end{array}
  \end{equation*}
  \item For $\lfloor a,a-2\rfloor,$ we have the following table:
  \begin{equation*}
    \begin{array}{ccc}
    s_c\circ \balpha(\lfloor a,a-2\rfloor), & \partial_c( \balpha(\lfloor a,a-2\rfloor)),&\textrm{Condition} \\
    \quad & \quad \\
    \alpha_{a-1}, & -2,& \textrm{if}\quad c=a-1 \\
    %\quad & \quad \\
    %\balpha(\lfloor a,b+1\rfloor), & -1,&\textrm{if}\quad c=b+1 \\
    \quad & \quad \\
    \balpha(\lfloor a+1,b\rfloor), &1,& \textrm{if}\quad c=a \\
    \quad & \quad \\
    \balpha(\lfloor a,a-3\rfloor), &1,& \textrm{if}\quad c=a-2 \\
    \quad & \quad \\
    \balpha(\lfloor a,a-2\rfloor), &0,& \textrm{otherwise} \\
    \quad & \quad
    \end{array}
  \end{equation*}
  \item For $\lceil a,a+2\rceil,$ we have the following table:
  \begin{equation*}
    \begin{array}{ccc}
    s_c\circ \balpha(\lceil a,a+2\rceil), & \partial_c( \balpha(\lceil a,a+2\rceil)),&\textrm{Condition} \\
    \quad & \quad \\
    \alpha_{a+1}, & -2,& \textrm{if}\quad c=a+1 \\
    %\quad & \quad \\
    %\balpha(\lfloor a,b+1\rfloor), & -1,&\textrm{if}\quad c=b+1 \\
    \quad & \quad \\
    \balpha(\lceil a-1,b\rceil), &1,& \textrm{if}\quad c=a \\
    \quad & \quad \\
    \balpha(\lceil a,a+3\rceil), &1,& \textrm{if}\quad c=a+2 \\
    \quad & \quad \\
    \balpha(\lceil a,a+2\rceil), &0,& \textrm{otherwise} \\
    \quad & \quad
    \end{array}
  \end{equation*}
  \end{enumerate}
\end{lemma}
\begin{proof}
  It follows directly by using repeatedly equations \ref{definition-dual-geometric-representation-TA}, \ref{Demazure-type-A-tilde-m} and \ref{eq-sum-of-roots-is-zero2}.
\end{proof}

In the following example we show how we can reduce the calculus of the matrix $\mathbf{T}_{\ulu},$ for a given expression, by using an decomposition in product of intervals and  the $\bnabla-$product of the corresponding extended matrices:
\begin{example}
  Let $m=5$ and $\ulu= s_0s_1s_2s_3s_2s_1=\lfloor 0,3\rfloor\lceil2,1\rceil.$ Thanks to lemma \ref{lemma-matrix-T-for-interval-any-type}, we already know that

  \begin{equation*}
    \widetilde{\mathbf{T}}_{\lfloor 0,3\rfloor}=\left[\begin{matrix}
                                                        \alpha_0 & 0 & 0 & 0 & 0 \\
                                                        \balpha(\lfloor 0,1\rfloor) & -1 & 0 & 0 & 0 \\
                                                        \balpha(\lfloor 0,2\rfloor) & -1 & -1 & 0 & 0 \\
                                                        \balpha(\lfloor 0,3\rfloor) & -1 & -1 & -1 & 0
                                                      \end{matrix}\right],\quad \widetilde{\mathbf{T}}_{\lceil 2,1\rceil}=\left[\begin{matrix}
                                                      \alpha_2 & 0 & 0 \\
                                                      \balpha(\lceil2,1\rceil) & -1 & 0
                                                    \end{matrix}\right].
  \end{equation*}
  therefore we only have to calculate the column vector $Q^{\lfloor 0,3\rfloor}_{\lceil2,1\rceil}$ and the matrix $\bpartial_{\lfloor0,3\rfloor}(Q_{\lceil2,1\rceil}).$

Using lemma \ref{lemma-tables-action-vs-Demazure}, we can check easily that
  \begin{equation*}
    {\lfloor 0,3\rfloor}\circ Q_{\lceil2,1\rceil}={\left[\begin{matrix}
                                          \alpha_3 \\
                                          \balpha(\lceil 3,2\rceil)
                                        \end{matrix}\right]},\quad
    \bpartial_{\lfloor0,3\rfloor}\left(Q_{\lceil2,1\rceil}\right)
                                        ={\left[\begin{matrix}
                                                      0 & 0 & +1 & -1 \\
                                                   0 & +1 & 0 & -1
                                                \end{matrix}\right]}
  \end{equation*}
   therefore

  \begin{equation*}
    \widetilde{\mathbf{T}}_{\ulu}=\widetilde{\mathbf{T}}_{\lfloor0,3\rfloor}\bnabla \widetilde{\mathbf{T}}_{\lceil2,1\rceil}
    =\left[\begin{matrix}
             \alpha_0                   & 0 & 0 & 0 & 0 & 0 & 0 \\
             \balpha(\lfloor0,1\rfloor) & -1 & 0 & 0 & 0 & 0 & 0 \\
             \balpha(\lfloor0,2\rfloor) & -1 & -1 & 0 & 0 & 0 & 0 \\
             \balpha(\lfloor0,3\rfloor) & -1 & -1 & -1 & 0 & 0 & 0 \\
             {\alpha_3}                   &{ 0} & {0} & {+1} & {-1} & 0 & 0 \\
             {\balpha(\lceil 3,2\rceil)} & {0} & {+1} & {0} & {-1} & -1 & 0
           \end{matrix}\right]
  \end{equation*}
  In particular, the presentation of the algebra $\bcalJ(\ulu)$ is given by generators $J_1,J_2,J_3,J_4,J_5,J_6$ and relations
  \begin{equation*}
    \left\{\begin{matrix}
      J_1^2= & 0, &J_2^2= & -J_1J_2\\
      \quad&\quad&\quad&\quad \\
      J_3^2= & -J_1J_3-J_2J_3,&J_4^2= & -J_1J_4-J_2J_4-J_3J_4 \\
       \quad&\quad&\quad&\quad \\
      J_5^2= & J_3J_5-J_4J_5,& J_6^2= & J_2J_6-J_4J_6-J_5J_6.
    \end{matrix} \right.
  \end{equation*}
\end{example}

%Note that
%\begin{equation*}
%  [Q_{\ulu},{\mathbf{T}}_{\ulu}]\bnabla[Q_{\ulw},{\mathbf{T}}_{\ulw}]
%=[Q^{\ulu}{\ulw},{\mathbf{T}}_{\ulu}\nabla_C{\mathbf{T}}_{\ulw}].
%\end{equation*}
%where

The next theorem is a very important fact that will be useful to develop our algorithm:

\begin{theorem}\label{lemma-commuting-move-for-type-Atilde}
  In type $\widetilde{A},$ the Gelfand-Tsetlin algebra $\bcalJ(\ulu)$ is an invariant under commuting moves in $W_m^{\ast}.$ That is, if $b,c\in I_m$ are such that $c\neq b,b\pm 1,$ $\ulu=s_{a_1}\cdots s_{a_{p-1}}s_bs_cs_{a_{p+2}}\cdots s_{a_n}$ and $\ulv=s_{a_1}\cdots s_{a_{p-1}}s_cs_bs_{a_{p+2}}\cdots s_{a_n}$ then the algebras $\bcalJ(\ulu)$ and $\bcalJ(\ulv)$ are isomorphic.
\end{theorem}

\begin{proof}
  Let $T=\mathbf{T}_{\ulu}=[t_{kj}]$ and $R=\mathbf{T}_{\ulv}=[r_{kj}].$ We denote by $\overrightarrow{T}_k$ and $\overrightarrow{R}_k$ the $k-$th row of the matrices $T$ and $R$ respectively and $\widehat{T}_j$ and $\widehat{R}_j$ the corresponding $j-$th column of the matrices $T$ and $R$ respectively.

  Note that by definition of the matrices $T=\mathbf{T}_{\ulu},R=\mathbf{T}_{\ulv}$ it is clear that $\overrightarrow{T}_k=\overrightarrow{R}_k$ for $k=1,\dots,p-1.$ In fact if $k=1,\dots,p-1,$ we have:
  \begin{equation*}
    t_{kj}=r_{kj}=\left\{\begin{array}{cc}
                           \partial_{a_j}(s_{a_{j+1}}\cdots s_{a_{k-1}}\circ \alpha_{a_k}) &\quad\textrm{if}\quad 1\leq j <k\leq p-1 \\
                           \quad & \quad \\
                           0 & \quad\textrm{if}\quad k\leq j
                         \end{array}\right.
  \end{equation*}
   Again by definition of $T=\mathbf{T}_{\ulu}$ and $R=\mathbf{T}_{\ulv}$ we have that $\widehat{T}_{j}=\widehat{R}_j$ for any $j\geq p+2.$ In fact if  $j\geq p+2$ we have:

  \begin{equation*}
    t_{kj}=r_{kj}=\left\{\begin{array}{cc}
                           \partial_{a_j}(s_{a_{j+1}}\cdots s_{a_{k-1}}\circ \alpha_{a_k}) &\quad\textrm{if}\quad p+2\leq j <k \\
                           \quad & \quad \\
                           0 & \quad\textrm{if}\quad k\leq j
                         \end{array}\right.
  \end{equation*}

  Since $s_b,s_c$ commutes as elements of the group $W_m,$ we also have that $t_{kj}=r_{kj}$ for any $k\geq p+2$ and $j=1,\dots,p-1.$  In fact if $k\geq p+2$ and $1\leq j\leq p-1$ we have:
  \begin{equation*}
    t_{kj}=\partial_{a_j}(s_{a_{j+1}}\cdots s_bs_c\cdots s_{a_{k-1}}\circ \alpha_{a_k})
  =\partial_{a_j}(s_{a_{j+1}}\cdots s_cs_b\cdots s_{a_{k-1}}\circ \alpha_{a_k})=r_{kj}.
  \end{equation*}

 Now note that $\overrightarrow{T}_{p}=\overrightarrow{R}_{p+1}$ and $\overrightarrow{T}_{p+1}=\overrightarrow{R}_{p}.$ In fact since $c\neq b, b\pm1$ we have $\partial_c(\alpha_b)=\partial_b(\alpha_c)=0.$ Therefore we can check that

 \begin{equation*}
  t_{pj}=r_{{p+1},j}=\left\{\begin{array}{cc}
                   \partial_{a_j}(s_{a_{j+1}}\cdots s_{a_{p-1}}\circ \alpha_{b}) & \quad \textrm{if}\quad 1\leq j\leq p-1 \\
                   \quad & \quad \\
                   0 & \textrm{otherwise}
                 \end{array} \right.
 \end{equation*}
 and
  \begin{equation*}
  r_{pj}=t_{{p+1},j}=\left\{\begin{array}{cc}
                   \partial_{a_j}(s_{a_{j+1}}\cdots s_{a_{p-1}}\circ \alpha_{c}) & \quad \textrm{if}\quad 1\leq j\leq p-1 \\
                   \quad & \quad \\
                   0 & \textrm{otherwise}
                 \end{array} \right.
 \end{equation*}
  Finally note that $\widehat{T}_p=\widehat{R}_{p+1}$ and $\widehat{R}_p=\widehat{T}_{p+1}.$ In fact

   \begin{equation*}
  t_{kp}=\left\{\begin{array}{cc}
                   \partial_{b}(s_{c}\circ h) & \quad \textrm{if}\quad p+2\leq k \\
                   \quad & \quad \\
                   0 & \textrm{otherwise}
                 \end{array} \right.\quad\textrm{and}\quad
  t_{k,{p+1}}=\left\{\begin{array}{cc}
                   \partial_{c}(h) & \quad \textrm{if}\quad p+2\leq k \\
                   \quad & \quad \\
                   0 & \textrm{otherwise}
                 \end{array} \right.
 \end{equation*}
 and
   \begin{equation*}
  r_{k,{p+1}}=\left\{\begin{array}{cc}
                   \partial_{b}(h) & \quad \textrm{if}\quad p+2\leq k \\
                   \quad & \quad \\
                   0 & \textrm{otherwise}
                 \end{array} \right.\quad\textrm{and}\quad
   r_{kp}=\left\{\begin{array}{cc}
                   \partial_{c}(s_{b}\circ h) & \quad \textrm{if}\quad p+2\leq k \\
                   \quad & \quad \\
                   0 & \textrm{otherwise}
                 \end{array} \right.
 \end{equation*}
 Where $h=s_{a_{p+2}}\cdots s_{a_{k-1}}\circ \alpha_{a_k}$ in both cases. By equation \ref{definition-dual-geometric-representation-TA}, we can write $h=\sum_{a\in I_m}{\lambda_a\alpha_a},$ for some $\lambda_a\in \F.$ Then by linearity and equations \ref{definition-dual-geometric-representation-TA} and \ref{Demazure-type-A-tilde-m} we have:
 \begin{equation*}
 \begin{array}{c}
    s_c\circ h=\sum_{a\in I_m}{\lambda_a(s_c\circ\alpha_a)}
   =\sum_{a\neq c}{\lambda_a\alpha_a}+(\lambda_{c+1}+\lambda_{c-1}-\lambda_c)\alpha_c. \\
   \quad \\
    s_b\circ h=\sum_{a\in I_m}{\lambda_a(s_b\circ\alpha_a)}
   =\sum_{a\neq b}{\lambda_a\alpha_a}+(\lambda_{b+1}+\lambda_{b-1}-\lambda_b)\alpha_c.
 \end{array}
 \end{equation*}
 and then

 \begin{equation*}
 \begin{array}{c}
    \partial_b(s_c\circ h)=\sum_{a\neq c}{\lambda_a\partial_b(\alpha_a)}
    +(\lambda_{c+1}+\lambda_{c-1}-\lambda_c)\partial_b(\alpha_c)=\sum_{a\neq c}{\lambda_a\partial_b(\alpha_a)}=\partial_b(h). \\
   \quad \\
    \partial_c(s_b\circ h)=\sum_{a\neq b}{\lambda_a\partial_c(\alpha_a)}
    +(\lambda_{b+1}+\lambda_{b-1}-\lambda_b)\partial_c(\alpha_b)=\sum_{a\neq b}{\lambda_a\partial_c(\alpha_a)}=\partial_c(h).
 \end{array}
 \end{equation*}

With all of this relations between the entries of the matrices $T$ and $R$ we have that the algebras $\calA(T)$ and $\calA(R)$ are defined as commutative algebras by the following presentations:

\begin{equation*}
  \calA(T):\left\{\begin{matrix}
                    X_1^2=0\quad \\
                    \quad \\
                    X_k^2=\sum_{j<k}t_{kj}X_jX_k,\quad(1\leq k<p)\\
                    \quad\\
                    X_p^2=\sum_{j<p}t_{pj}X_jX_p\\
                    \quad \\
                    X_{p+1}^2=\sum_{j<p}t_{{p+1}j}X_jX_{p+1}\\
                    \quad\\
                    X_k^2=\sum_{j\neq p,p+1}t_{kj}X_jX_k
                    +t_{kp}X_pX_k+t_{k,{p+1}}X_{p+1}X_k, \quad (p+1<k<n)\\
                  \end{matrix}\right.
\end{equation*}
and

\begin{equation*}
  \calA(R):\left\{\begin{matrix}
                    Y_1^2=0 \\
                    \quad \\
                    Y_k^2=\sum_{j<k}t_{kj}Y_jY_k,\quad(1\leq k<p)\\
                    \quad\\
                    Y_p^2=\sum_{j<p}t_{{p+1},j}Y_jY_p\\
                    \quad \\
                    Y_{p+1}^2\sum_{j<p}t_{{p}j}Y_jY_{p+1}\\
                    \quad \\
                    Y_k^2=\sum_{j\neq p,p+1}t_{kj}Y_jY_k
                    +t_{k,{p+1}}Y_pY_k+t_{k,{p}}Y_{p+1}Y_k, \quad (p+1<k<n)\\
                  \end{matrix}\right.
\end{equation*}

and therefore it is clear that the assignment:
\begin{equation*}
  \bgamma:\left\{\begin{array}{cc}
                  X_k\mapsto Y_k,& (k\neq p,p+1)  \\
                  X_p\mapsto Y_{p+1} & \quad\\
                  X_{p+1}\mapsto Y_p & \quad
                \end{array}\right.
\end{equation*}
defines an isomorphism  $\calA(T)\rightarrow \calA(R).$ Note that the matrix associated $\Gamma$ is invertible, since $\det(\Gamma)=-1\neq 0.$
The conclusion follows from theorem \ref{theo-Ju-isomorphic-to-ATu} and definition \ref{def-nilalg-associted-to-expression}.

%In fact, one can check:
%\begin{enumerate}
%  \item $\gamma(X_1^2)=\gamma(X_1)^2=0.$
%  \item If $1<k<p$
%  \begin{equation*}
%    \gamma(X_k^2)=\sum_{j<k}{t_{kj}\gamma(X_j)\gamma(X_k)}
%    =\sum_{j<k}{r_{kj}Y_jY_k}=\gamma(X_k)^2.
%  \end{equation*}
%  \item
%  \begin{equation*}
%   \gamma(X_p^2)=\sum_{j<p}{t_{pj}\gamma(X_j)\gamma(X_p)}
%   =\sum_{j<p}{r_{{p+1},j}Y_jY_{p+1}}=\gamma(X_p)^2.
%   \end{equation*}
%   Since $r_{{p+1},p}=0.$
%   \item
%   \begin{equation*}
%   \gamma(X_{p+1}^2)=\sum_{j<{p+1}}{t_{{p+1}j}\gamma(X_j)\gamma(X_{p+1})}
%   =\sum_{j<p}{r_{{p},j}Y_jY_{p}}=\gamma(X_{p+1})^2.
%   \end{equation*}
%   Since $t_{{p+1},p}=0.$
%   \item If $k>p+1$ then
%   \begin{equation*}
%     \gamma(X_k^2)=\sum_{j\neq p,p+1}{t_{kj}\gamma(X_j)\gamma(X_k)}
%     +t_{kp}\gamma(X_p)\gamma(X_k)+t_{k,{p+1}}\gamma(X_{p+1})\gamma(X_k)
%   \end{equation*}
%   and
%   \begin{equation*}
%     \gamma(X_k)^2=\sum_{j\neq p,p+1}{r_{kj}Y_jY_k}
%     +r_{k{p+1}}Y_{p+1}Y_k+r_{k,{p}}Y_{p}Y_k.
%   \end{equation*}
%   then $\gamma(X_k^2)=\gamma(X_k)^2.$
%\end{enumerate}
\end{proof}

We now define the following algorithm to choose a representative expression $\ulu'$ for $\ulu,$ equivalent by commuting moves:

\begin{definition}\label{def-abacus}
  Let $\ulu=s_{a_1}s_{a_2}\cdots s_{a_n}\in W_{m}^{\ast}.$ If $a_j=a_i,a_i\pm 1,$ we say that the symbols ${a_j}$ and $a_{i}$ are neighbors. We define the \emph{abacus of letters} of $\ulu$ as the array $\Xi=\Xi(\ulu)$ described by the following construction:
  \begin{enumerate}
    \item In the first line of $\Xi$ we put the symbol ${a_1}.$ We say at this point that the first line of $\Xi$ have both an \emph{ascending pattern} and a \emph{descending pattern}. In the following steps we say that a line composed by just one symbol have both an \emph{ascending pattern} and a \emph{descending pattern}.
    \item
    \begin{itemize}
      \item If $a_2=a_{1}+1$ we say that the symbol $a_2$ respect the \emph{ascending pattern} of line 1 in $\Xi$ and we put the symbol ${a_2}$ at the right of ${a_1}$ (same line). In this case we say that the (new) line 1 of $\Xi$ take the \emph{ascending pattern}
      \item If $a_2=a_{1}-1$ we say that the symbol $a_2$ respect a \emph{descending pattern} of line 1 of $\Xi$ and we put the symbol ${a_2}$ at the right of ${a_1}$ (same line). In this case we say that the (new) line 1 of $\Xi$ take the \emph{descending pattern}.
      %\item If $a_2=a_{1}$ then we put the symbol $a_2$ at the second position of the second line of $\Xi$ (going down and right from $a_1$), we also put the symbol $\emptyset$ in the first position of the second line  of $\Xi.$
      \item Otherwise we put the symbol $a_2$ in the first position of the second line of $\Xi.$
    \end{itemize}
    \item Assuming that we already have located the symbols ${a_1},\cdots,{a_p}$ in $\Xi$ using the first $q$ lines of it. We also assume that any of those lines have a defined pattern: ascending or descending (or both in case that there is only one symbol in that line).  Let ${b_1},\cdots,{b_q}$ the corresponding rightmost symbol located in lines $1,\dots,q$ in $\Xi.$ For the symbol ${a_{p+1}}$ we have:
    \begin{itemize}
      \item If ${a_{p+1}}=b_q\pm 1,$ and the symbol respect the pattern of line $q,$ we put the symbol $a_{p+1}$ in the same line just at the right of the symbol $b_q.$ The (new) line keep (take) its pattern.
      \item If ${a_{p+1}}\neq b_q\pm 1,$ but there is a neighbor symbol of ${a_{p+1}}$ in line $q.$ Then we put the symbol $a_{p+1}$ at the first position of line $q+1.$
          %Then we fill with $emptyset$ all the positions of line $q$ from the leftmost until the position just down of the neighbor symbol. Then we put the symbol $a_{p+1}$ in line $q+1$ in the next position available. (This case include the situation when ${a_{p+1}}=b_q$).
      \item If there is no neighbor symbols of ${a_{p+1}}$ in lines $j+1,\dots,q$ but there is a neighbor symbols of ${a_{p+1}}$ in line $j$ then we apply the same analysis:
          \begin{itemize}
            \item If ${a_{p+1}}=b_j\pm 1.$ In this case, if also the symbol respect the pattern of line $q,$ we put the symbol $a_{p+1}$ in the same line just at the right of the symbol $b_j.$ The new line keeps its pattern.
            \item If ${a_{p+1}}\neq b_j\pm 1,$ but there is a neighbor symbol of ${a_{p+1}}$ in line $j.$ Then we put the symbol $a_{p+1}$ at the first position of line $q+1.$
                %Then we fill with $emptyset$ all the positions of line $q$ from the leftmost until the position corresponding position of the neighbor symbol. Then we put the symbol $a_{p+1}$ in line $q+1$ in the next position available. (This case include the situation when ${a_{p+1}}=b_j$).
          \end{itemize}
      \item Otherwise we put the symbol in the first (leftmost) position of line $q+1.$
    \end{itemize}
  \end{enumerate}
  For each $j=1,\dots f$ if the $j-$th line of the abacus $\Xi,$ is given by the symbols $c_{j1}\cdots c_{jg},$ we define the expression  $L_j=s_{c_{j1}}\cdots s_{c_{jg}}\in W_m^{\ast}.$ Finally we define the expression $\ulu'=L_1\cdots L_f\in W_m^{\ast}.$
\end{definition}

\begin{example}
  Let $m=10$ and $\ulu=s_3s_1s_0s_4s_2s_8s_7s_5s_6s_2s_5s_1s_0.$ Following our algorithm we obtain at the end:
  \begin{equation*}
    \Xi=\left\{\begin{matrix}
                 3 & 4& 5& \quad \\
                 1 & 0& \quad& \quad \\
                 2 & \quad& \quad& \quad \\
                 8 & 7& 6& 5 \\
                 2&1 & 0& \quad&
               \end{matrix}\right.
  \end{equation*}
  Then we obtain the expression $\ulu'=\lfloor3,5\rfloor\lceil1,0\rceil s_2 \lceil8,5\rceil \lceil2,0\rceil.$
\end{example}

%\begin{lemma}\label{GT-alg-for-commuting-expressions}
%  Let $\ulu,\ulv\in W^{\ast}_m$ reduced expressions of $u,v\in W_m$ respectively. If $u$ and $v$ commute, and $\ulw=\ulu\ulv$ is a reduced expression of $uv,$ then the algebra $\bcalJ(\ulw)$ is isomorphic to $\bcalJ(\ulu)\nabla_0\bcalJ(\ulv).$
%\end{lemma}
%\begin{proof}
%  If $\ulu=s_{a_1}\cdots s_{a_{p}}$ and $\ulv=s_{b_1}\cdots s_{b_{q}},$ the by our hypothesis we can assume that $a_i\neq b_j,b_j\pm 1$ for any pair $i,j.$
%\end{proof}

%As we learned in section \ref{ssec-the-Jailbreak-alg}, for each expression $\ulu\in W_{m}^{\ast},$ the algebra $\bcalJ_0(\ulu)$ is totally determined by the corresponding matrices $\mathbf{T}_{\ulu}$ and $\widetilde{\mathbf{T}}_{\ulu}.$ In the following we provide relevant information that will allow defining algorithms for the explicit calculation of all these matrices

\begin{theorem}\label{theo-summarizing}
  Given an expression $\ulu\in W_{m}^{\ast},$ and let $\ulu'=L_1\cdots L_f,$ as in definition \ref{def-abacus}. Then we have
  \begin{enumerate}
    \item The algebras $\bcalJ(\ulu)$ and $\bcalJ(\ulu')$ are isomorphic.
    \item The factors $L_1,\dots,L_f$ are intervals (in the sense of equation \ref{eq-up-interval-coxeter}).
    \item The algebra $\bcalJ(\ulu)$ is isomorphic to the nil graded algebra $\calA(T_{L_1}\bnabla\cdots \bnabla T_{L_{f}}).$
  \end{enumerate}
\end{theorem}
\begin{proof}
  For the first assertion, note that we obtain the expression $\ulu'$ from $\ulu,$ applying a finite sequence of commuting moves. The conclusion follows from theorem \ref{lemma-commuting-move-for-type-Atilde}.

  The second assertion follows directly from definition \ref{def-abacus}.

  Finally the third assertion follows from lemmas \ref{lemma-special-nabla-product-for-calA-uw}, \ref{lemma-associativity-bnabla-product} and corollary \ref{coro-theo-Ju-isomorphic-to-ATu}
\end{proof}

Note that in particular, Theorem \ref{theo-summarizing} implies that for any expression $\ulu,$ the algebra $\bcalJ(\ulu)$ comes endowed with an optimal presentation codified by the matrix $T_{L_1}\bnabla\cdots \bnabla T_{L_{f}}.$
%Finally by corollary \ref{coro-theo-Ju-isomorphic-to-ATu}, to obtain a presentation for the algebra $\bcalJ(\ulu),$ we only need is to obtain presentations for any interval $\ulw$ and then to learn how to connect them via $\nabla-$products.

\begin{corollary}
  For any expression $\ulu\in W_m^{\ast},$ all the entries of the associated matrix $\mathbf{T}_{\ulu}$ are elements of the set $\{0,1,-1,-2\}.$
\end{corollary}
\begin{proof}
  It follows directly from theorem \ref{theo-summarizing} and lemma \ref{lemma-tables-action-vs-Demazure}.
\end{proof}

\begin{theorem}\label{theo-isomorph-for-vertical-case}
  Let $\ulu=\lfloor a,a-1\rfloor^{h}=\lfloor a,a-1\rfloor\cdots \lfloor a,a-1\rfloor,$ ($h-$factors). Then the algebra $\bcalJ(\ulu)$ is isomorphic to the Gelfand-Tsetlin subalgebra of the vertical idempotent truncation $\boldsymbol{\calB}(\underline{\boldsymbol{i}})$ of the generalized blob algebra $\boldsymbol{\calB}=\boldsymbol{\calB}_{l,n}^{\F}(e)$ with $l=m+1$ and $n=he$ (see \cite{Lobos1} for notation).
\end{theorem}

\begin{proof}
  From \cite{Lobos1}, we know that the Gelfand-Tsetlin subalgebra of the vertical idempotent truncation $\boldsymbol{\calB}(\underline{\boldsymbol{i}})$ of the generalized blob algebra $\boldsymbol{\calB},$ is an algebra $\mathbb{D},$ defined by generators:

  \begin{equation*}
    \mathbb{L}_{(r,j)},\quad (1\leq r \leq h,0 \leq j \leq l-2)
  \end{equation*}

  and relations
  \begin{equation*}
    (\mathbb{L}_{(r,j)})^2=\sum_{(t,i)<(r,j)}{C_{(t,i)}\mathbb{L}_{(t,i)}\mathbb{L}_{(r,j)}},
  \end{equation*}

  where $(t,i)<(r,j)$ is the lexicographic order and

  \begin{equation*}
    C_{(t,i)}=\left\{\begin{array}{cc}
                      -2 & \quad \textrm{if}\quad i=j \\
                      -1 & \quad \textrm{otherwise}
                    \end{array}\right.
  \end{equation*}

  Since the lexicographic order is a total order on $\{(r,j):1\leq r \leq h,0 \leq j \leq m-1\}$ we can re enumerate increasingly the elements $\mathbb{L}_{(r,j)}$ as $\mathbb{L}_1,\mathbb{L}_2,\dots\mathbb{L}_{hm}.$ With this notation we can check that the algebra $\mathbb{D}$ is a nil-graded algebra associated to the matrix $T=[t_{kj}]$ given by
  \begin{equation*}
    t_{kj}=\left\{\begin{array}{cc}
                    0 & \quad \textrm{if}\quad k\geq j \\
                    -2 & \quad \textrm{if}\quad k-j\in (m-1)\mathbb{Z} \\
                    -1 & \quad \textrm{otherwise}
                  \end{array}\right.
  \end{equation*}
  We also know from \cite{Lobos1}, that $\dim(\mathbb{D})=2^{hm},$ therefore the algebra $\mathbb{D}$ is a nil graded algebra strongly associated to the matrix $T,$ then it is isomorphic to the algebra $\calA(T).$

  The isomorphism with $\bcalJ_0(\ulu)$ follows from lemmas \ref{lemma-matrix-T-for-interval-any-type}, \ref{lemma-tables-action-vs-Demazure}, and \ref{lemma-associativity-bnabla-product}. Basically one can check that the extended matrix $\widetilde{\mathbf{T}}_{\lfloor a, a-1\rfloor^{h}}$ looks like:

  \begin{equation*}
    \left[\begin{matrix}
     \alpha_a & 0 & 0 & \cdots & \cdots & 0& 0& \cdots\\
     \balpha(\lfloor a, a+1\rfloor) & -1 & 0 & \cdots & \cdots & 0 & 0& \cdots \\
      \vdots & \vdots & \ddots & \ddots & \cdots & \vdots& \vdots& \cdots \\
       \balpha(\lfloor a, a-2\rfloor) & -1 & -1 & \ddots & 0 & 0 & 0& \cdots\\
       \alpha_a & {\color{blue}-2} & -1 & \ddots & -1 & 0 & 0& \cdots\\
       \balpha(\lfloor a, a+1\rfloor) & -1 &  {\color{blue}-2} & -1 & \ddots & -1 & 0& \cdots \\
       \vdots & \vdots & \ddots & \ddots & \ddots & \ddots& \ddots & \ddots\\
       \alpha_a & {\color{blue}-2} & -1 & \cdots &  {\color{blue}-2} &  -1 & -1& \ddots\\
       \balpha(\lfloor a, a+1\rfloor) & -1 &  {\color{blue}-2} & -1 & \cdots & {\color{blue}-2} & -1& \ddots\\
       \vdots & \vdots & \ddots & \ddots & \ddots & \ddots& \ddots& \ddots \\
       \end{matrix}\right]..
  \end{equation*}

  In particular the matrix $\mathbf{T}_{\ulu}=[t_{kj}]$ is given by
  \begin{equation*}
    t_{kj}=\left\{\begin{array}{cc}
                    0 & \quad \textrm{if}\quad k\geq j \\
                    -2 & \quad \textrm{if}\quad k-j\in (m-1)\mathbb{Z} \\
                    -1 & \quad \textrm{otherwise}
                  \end{array}\right.
  \end{equation*}
  Therefore the algebras $\bcalJ(\ulu)$ and $\mathbb{D}$ are nil graded algebras strongly associated to the same matrix, and then they are isomorphic.
\end{proof}

\begin{remark}
  The case when $m=2$ correspond to the \emph{Dot-line algebra} developed by Espinoza and Plaza in their article \cite{Esp-Pl}. In that case note that all the not zero entries in the matrix $\mathbf{T}_{\ulu}$ are equal to $-2.$
\end{remark}

\sc
diego.lobos@pucv.cl, Pontificia Universidad Cat\'olica de Valpara\'iso, Chile.


\begin{thebibliography}{X}

\bibitem{AlHar} S. Al Harbat, \textit{Canonical reduced expression in affine Coxeter groups PartI -Type $\widetilde{A}_{n}$}, arXiv:2105.07417

% \bibitem{A}
% S. Ariki, \textit{On the decomposition numbers of the Hecke algebra of $G(m, 1, n)$}, J.
% Math. Kyoto Univ. {\bf 36} (1996), 789-808.

\bibitem{Bjorner-Brenti} A. Björner, F. Brenti,
\textit{Cobinatorics of Coxeter Groups}, Graduate Text in Mathematics, Springer Science+Business Media, Inc. 2005.

%\bibitem{bowman} C. Bowman,  \textit{The many integral graded cellular bases of Hecke algebras of complex reflection groups}, American Journal of Mathematics, DOI: 10.1353/ajm.2022.0008.

\bibitem{bow-cox-hazi} C. Bowman, A. Cox, A. Hazi,  \textit{Path isomorphisms between quiver Hecke and diagrammatic Bott-Samelson endomorphism algebras},
  arXiv:2005.02825v3.

%\bibitem{bcs} C. Bowman, A. Cox, L. Speyer,  \textit{A Family of Graded Decomposition Numbers for Diagrammatic
 %  Cherednik Algebras},
 %IMRN {\bf 2017}(9) (2017), 2686-2734.



% \bibitem{BJ} C. Bonnaf\'e, N. Jacon,
% \textit{Cellular structures on Hecke algebras of type $ B$},
% J. Algebra, {\bf 321}, Issue 11, 2009, 3089-3111.

% \bibitem{BGIL} C. Bonnaf\'e, M. Geck, L. Iancu, T. Lam,  \textit{On domino insertion and
 %  Kazhdan-Lusztig cells in type $B_n$}. In: Representation theory of algebraic groups and quantum groups
 %(Nagoya, 2006; eds. A. Gyoja et al.), Progress in Math., {\bf 284}, Birkh\"auser/Springer, New York, 2010, 3354


\bibitem{brundan-klesc} J. Brundan, A. Kleshchev,
\textit{Blocks of cyclotomic Hecke algebras and Khovanov-Lauda algebras}, Invent. Math. {\bf 178} (2009), 451-484.

% \bibitem{DJM} R. Dipper, G. James, A. Mathas,
% \textit{Cyclotomic q-Schur algebras}, Math. Z., {\bf 229} (1998), 385-416.


%\bibitem{blob positive} A. Cox, J. Graham, P. Martin (2003), \textit{The blob algebra in positive characteristic},
%Journal of Algebra, {\bf 266}(2), 584-635.

\bibitem{EliasKhov10} B. Elias, M. Khovanov, {\it Diagrammatics for Soergel categories},  Int. J. Math. Math. Sci.
(2010), Art. ID 978635,58.

\bibitem{EW} B. Elias, G. Williamson, {\it Soergel calculus},  Representation Theory
{\bf 20} (2016), 295-374.

\bibitem{Esp-Pl} J. Espinoza, D. Plaza,
\textit{Blob algebra and two-color Soergel calculus},
Journal of Pure and Applied Algebra {\bf 223}(11), (2019), 4708-4745.




% \bibitem{FLOTW} O. Foda, B. Leclerc, M. Okado, J.-Y. Thibon, T. Welsh, \textit{Branching functions of $A^{(1)}_{n-1} $
% and Jantzen-Seitz problem for Ariki-Koike algebras}, Advances in Mathematics {\bf 141} (1999), 322-365.

% \bibitem{Geck} M. Geck, \textit{Hecke algebras of finite type are cellular}, Inventiones mathematicae {\bf 169}, (2007),  501-517.

 \bibitem{Gelf-Tse1} I.M. Gelfand, M.L. Tsetlin, \textit{Finite-dimensional representation of the group of unimodular matrices},
  Dockl. Akad. Nauk SSSR {\bf 71}, (1950),  825-828.

 \bibitem{Gelf-Tse2} I.M. Gelfand, M.L. Tsetlin, \textit{Finite-dimensional representation of the group of othogonal matrices},
  Dockl. Akad. Nauk SSSR {\bf 71}, (1950),  1017-1020.



\bibitem{GL} J. J. Graham, G. I. Lehrer,
  \textit{Cellular algebras}, Inventiones Mathematicae {\bf 123} (1996), 1-34.


%\bibitem{HMP} A. Hazi, P. Martin, A. Parker, \textit{Indecomposable tilting modules for the blob algebra},
%arXiv:1809.10612.


\bibitem{H1} J. E. Humphreys, {\it Reflection groups and Coxeter groups}, volume {\bf 29} of Cambridge Studies in
Advanced Mathematics. Cambridge University Press, Cambridge, 1990.





%\bibitem{hu-mathas} J. Hu, A. Mathas, \textit{Graded cellular bases for the cyclotomic Khovanov-Lauda-Rouquier
%algebras of type $A$}, Adv. Math., {\bf 225} (2010), 598-642.


\bibitem{JenWill17} L.T.Jentsen, G. Williamson,  \textit{The $p-$Canonical basis for Hecke Algebras} in Categorification in Geometry, Topology and Physics, 333-361, Contemp. Math. 583 (2017).

% \bibitem{JMMO} M. Jimbo, K. Misra, T. Miwa, M. Okado,
%   \textit{Combinatorics of representations of $U_q( \hat{sl}_n)$ at $q = 0$}
% Comm. Math. Phys. {\bf 136} (1991), 543-566.


% \bibitem{JaKe} G. James,  A. Kerber, \textit{The Representation Theory of the Symmetric Group}. Encyclopedia of
% Math. and its Applications (Addison-Wesley, Reading, MA, 1981).

 \bibitem{Jucys} A. Jucys, \textit{Symmetric polynomials and the center of the symmetric group ring }, Reports Math. Phys., {\bf 5} (1974), 107-112.

%\bibitem{KhovanovLauda} M. Khovanov, A. Lauda, \textit{A diagrammatic approach to categorification of quantum groups I}, Represent. Theory {\bf 13}
%(2009), 309-347.





% \bibitem{GLi} G. Li, \textit{Integral basis theorem of cyclotomic
 %Khovanovâ€“Laudaâ€“Rouquier algebras of type $A$}, Journal of Algebra {\bf 482} (2017), 1-101.


\bibitem{LibLightLeaves} N. Libedinsky, {\it Sur la categorie des bimodules de Soergel},
J. Algebra {\bf 320 }(7) (2008), 2675-2694.

\bibitem{Lib10} N. Libedinsky, {\it Presentation of the right-angled Soergel categories by generators and relations},
J. Pure Appl. Algebra {\bf 214 }(2010), no. 12, 2265-2278.

\bibitem{Lib15} N. Libedinsky, {\it Light leaves and Lusztig's conjecture},
Adv. Math. {\bf 280 }(2015), 722-807.

\bibitem{LibGentle} N. Libedinsky, {\it
  Gentle introduction to Soergel bimodules I: The basics}, {\color{black}{Sao Paulo Journal of
  Mathematical Sciences, {\bf 13}(2) (2019), 499-538.}}



\bibitem{LiPl} N. Libedinsky, D. Plaza, \textit{Blob algebra approach to modular representation theory},
Proc. Lond. Math. Soc. (3) {\bf 121} (2020) no. 3, 656-701.

\bibitem{Lobos1} D. Lobos, \textit{On generalized blob algebras: Vertical idempotent truncations and Gelfand-Tsetlin subalgebras}, arXiv:2203.15139.


\bibitem{Lobos-Plaza-Ryom-Hansen} D. Lobos, D. Plaza S. Ryom-Hansen, \textit{The Nil-blob algebra: An incarnation of type $\tilde{A}_1$ Soergel calculus and of the truncated blob algebra},
J. Algebra {\bf 570} (2021), 297-365.

\bibitem{Lobos-Ryom-Hansen} D. Lobos, S. Ryom-Hansen, \textit{Graded cellular basis and Jucys-Murphy elements
  for generalized blob algebras},
Journal of Pure and Applied Algebra {\bf 224} (7), (2020), 106277.


% \bibitem{Lu} G. Lusztig, \textit{Hecke algebras with unequal parameters}, CRM Monographs
% Ser. {\bf 18}, Amer. Math. Soc., Providence, RI, 2003.

\bibitem{MacLane1} S. Mac Lane \textit{Categories for the Working Mathematician}, 2nd edition, Graduate Text in Mathematics, Springer Science+Business Media, 1998.

\bibitem{MacLane2} S. Mac Lane \textit{Homology}, Classics in Mathematics, Springer-Verlag Berlin Heidelberg, 1995.

%\bibitem{Mat-Sal} P. P. Martin, H. Saleur, \textit{The blob algebra and the periodic Temperley-Lieb algebra},
%Lett. Math. Phys. {\bf 30} (1994), 189-206.

%\bibitem{MW} P. P. Martin,  D. Woodcock, \textit{Generalized blob algebras and alcove geometry}, LMS Journal
%of Computation and Mathematics {\bf 6}, (2003), 249-296.

%\bibitem{martin-wood1} P. P. Martin, D. Woodcock, \textit{On the structure of the blob algebra}, J. Algebra {\bf 225} (2000), 957-988.

 \bibitem{MatCoef} A. Mathas, \textit{Matrix units and generic degrees for the Ariki Koike algebras},
 J. of Algebra {\bf 281} (2004) 695-730.



% \bibitem{Mat} A. Mathas,
% \textit{Hecke algebras and Schur algebras of the symmetric group},
% Univ. Lecture Notes, 15, A.M.S., Providence, R.I., 1999.

\bibitem{Mat-So} A. Mathas, \textit{Seminormal forms and Gram determinants for cellular algebras}, J.
Reine Angew. Math., {\bf 619} (2008), 141-173.  With an appendix by M. Soriano.



 \bibitem{Murphy2} G. E. Murphy, \textit{A new construction of Young's seminormal representation of the symmetric groups},
 J. of Algebra {\bf 69} (1981), 287-291.


 \bibitem{Murphy} G. E. Murphy, \textit{The idempotents of the symmetric group and
 Nakayama's conjecture}, J. of Algebra {\bf 81} (1983), 258-265.

 \bibitem{Murphy1} G. E. Murphy, \textit{The Representations of Hecke Algebras of type
 $ A_n $}, J. of Algebra {\bf 173} (1995), 97-121.

 \bibitem{Murphy3} G. E. Murphy, \textit{On the Representation Theory of the Symmetric Groups and associated Hecke Algebras},
 J. of Algebra {\bf 152} (1992), 492-513.


 \bibitem{OkunVershik1} A. Okounkov, A. Vershik \textit{A new approach to representation theory of symmetric group},
 Selecta Math. {\bf 2} (4), (1996), 581-605.

 \bibitem{OkunVershik2} A. Okounkov, A. Vershik \textit{A new approach to Representation Theory of the Symmetric Groups. II},
 J. Math. Sci. {\bf 131} (2005), 5471-5494.

\bibitem{RicheWill} S. Riche, G. Williamson, \textit{Tilting Modules and The $p$-Canonical Basis},
 Ast\'erisque {\bf 397} (2018), 1-184.

\bibitem{Plaza17} D. Plaza, \textit{Graded cellularity and the Monotonicity Conjecture}, J. Algebra, {\bf 473} (2017), 324-351.

%\bibitem{Plaza13}
%D. Plaza. \textit{Graded decomposition numbers for the blob algebra},
%J. Algebra, {\bf 394} (2013), 182-206.

%\bibitem{PlazaRyom}
%D. Plaza,  S. Ryom-Hansen, \textit{Graded cellular bases for Temperley-Lieb algebras of type A and B},
%Journal of Algebraic Combinatorics, {\bf 40}(1) (2014), 137-177.

%\bibitem{Rouq} R. Rouquier, \textit{2-Kac-Moody Algebras}, arXiv:0812.5023.

%\bibitem{steen} S. Ryom-Hansen, \textit{Jucys-Murphy operators for Soergel bimodules},
%arXiv:1607.03470.




\bibitem{So} W. Soergel, {\it Kazhdan-Lusztig-Polynome und unzerlegbare Bimoduln \"uber
Polynomringen}, J. Inst. Math. Jussieu {\bf 6} (2007), no. 3, 501--525.


\bibitem{So1} W. Soergel. {\it Kategorie $\mathcal O$, perverse Garben und Moduln \"uber den Koinvarianten
 zur Weylgruppe}. J. Amer. Math. Soc., {\bf 3}(2) (1990), 421--445,


% \bibitem{Steen1} S. Ryom-Hansen, \textit{Cell structures on the blob algebra}, Representation Theory {\bf 16} (2012), 540-567.

 \bibitem{Steen2} S. Ryom-Hansen, \textit{Jucys-Murphy elements for Soergel Bimodules}, J. Algebra {\bf 551} (2020), 154-190.

 %\bibitem{Uglov} D. Uglov, \textit{Canonical bases of higher level $q$-deformed Fock spaces and
 %Kazhdan-Lusztig polynomials}, Kashiwara, Masaki (ed.) et al., Boston: Birkh\"auser. Prog. Math. {\bf 191} (2000):
 %249-299.


 \bibitem{W} B. Webster, \textit{Rouquier's conjecture and diagrammatic algebra}, Forum Math. Sigma {\bf 5}
   (2017), e27, 71.

 \bibitem{Westbury} B. W. Westbury, \textit{Invariant Tensors and cellular cateogories}, J. Algebra {\bf 321}(11) 2009, 3563-3567.

\end{thebibliography}
\end{document}